\documentclass[11pt]{article}
\usepackage{amssymb, latexsym, mathtools, amsthm, amsmath}
\begingroup\makeatletter
\@for\theoremstyle:=definition,remark,plain\do{%
\expandafter\g@addto@macro\csname th@\theoremstyle\endcsname{%
\addtolength\thm@preskip\parskip}}
\endgroup
\usepackage{authblk}
\usepackage{stmaryrd}

\usepackage{geometry}
 \renewenvironment{abstract}
 { \normalsize
  \list{}{\setlength{\leftmargin}{.0cm}%
    \setlength{\rightmargin}{\leftmargin}}%
  \item {\bf \abstractname.} \relax}
 {\endlist}
\usepackage{titlesec}
 \titleformat{\paragraph}
{\normalfont\normalsize\em\bfseries}{\theparagraph}{1em}{$\blacktriangleright$\hspace{0.2cm}}
\titlespacing*{\paragraph}{0pt}{3.25ex plus 1ex minus .2ex}{0.5ex plus .2ex}
\usepackage{txfonts}
\usepackage{booktabs}
\usepackage{pgf,tikz}
\definecolor{dnrbl}{rgb}{0,0,0.3}
\definecolor{dnrgr}{rgb}{0,0.3,0}
\definecolor{dnrre}{rgb}{0.5,0,0}
\usepackage[colorlinks=true, citecolor=dnrgr, linkcolor=dnrre, urlcolor=dnrre] {hyperref}
\usepackage{xcolor,colortbl}
\usepackage{enumerate}
\usepackage{pdfsync}
\usepackage[numbers]{natbib}
\usepackage{datetime}
\usepackage{parskip}
\usepackage{tabularx}
\theoremstyle{plain}
\newtheorem{thm}{Theorem}[section]

\newtheorem{lem}[thm]{Lemma}
\newtheorem{coro}[thm]{Corollary}
\newtheorem*{coron}{Corollary}

\theoremstyle{definition}

\newtheorem{defi}[thm]{Definition}
\newcommand{\Nat}{\mathbb{N}}

\newcommand{\restr}{\upharpoonright}  %restriction

\newcommand{\un}{\uparrow} %undefined
\newcommand{\de}{\downarrow} %defined

\DeclarePairedDelimiter{\tuple}{\langle}{\rangle}

\newcommand{\bigo}[1]{\mathop{\sc\textup O}\/\left({#1}\right)}

\newcommand{\sqbradn}[2]{\{\hspace{0.02cm}{#1} : {#2}\hspace{0.02cm}\}}
\newcommand{\sqbrad}[2]{\big\{\hspace{0.02cm}{#1} : {#2}\hspace{0.02cm}\big\}}
\newcommand{\sqbradlr}[2]{\left\{\hspace{0.02cm} {#1} : {#2}\hspace{0.02cm}\right\}}

\newcommand{\sqbradb}[2]{\big\{\ {#1} : {#2} \big\}}
\DeclarePairedDelimiter{\dbra}{\llbracket}{\rrbracket}

\newcommand{\etal}{\textit{et al}.}

\newcommand{\EE}{\mathcal{E}}

\newcommand{\II}{\mathcal{I}}

\newcommand{\DD}{\mathcal{D}}

\newcommand{\sqbradB}[2]{\Big\{\ {#1}\ :\ {#2}\ \Big\}}

\newcommand{\wgt}[1]{\mathop{\mathtt{wgt}}\/\left({#1}\right)}
\newcommand{\abs}[1]{|{#1}|}

\newcommand{\absb}[1]{\big|{#1}\big|}
\newcommand{\parb}[1]{\big({#1}\big)}
\newcommand{\parB}[1]{\Big(\hspace{0.04cm}{#1}\hspace{0.04cm}\Big)}
\newcommand{\parlr}[1]{\left({#1}\right)}

\newcommand{\ewt}[1]{\mathop{\mathtt{ewt}}\/\left({#1}\right)}

\newcommand{\LDF}{\mathtt{LDF}}

\newcommand{\TTast}{\mathcal{T}^{\ast}}
\newcommand{\TT}{\mathcal{T}}

\newcommand{\Tb}{\mathbf{T}}

\newcommand{\ml}{Martin-L\"{o}f }

\newcommand{\pz}{$\Pi^0_1$\ }

\newcommand{\ce}{c.e.\ }

\newcommand{\pf}{prefix-free }
\newcommand{\pfn}{prefix-free}

\newcommand{\wedga}{\ \wedge\ }

\newcommand{\VV}{\mathcal V}

\newcommand{\FF}{\mathcal F}

\newcommand{\twome}{2^{\omega}}
\newcommand{\twomel}{2^{<\omega}}

\newcommand{\cost}[1]{\mathsf{c}\hspace{0.03cm}({#1})}

\newcommand{\RCA}{\mathsf{RCA}}
\newcommand{\WKL}{\mathsf{WKL}}

\newcommand{\WWKL}{\mathsf{WWKL}}
\newcommand{\PPp}{\mathsf{P}^+}
\newcommand{\PPm}{\mathsf{P}^-}
\newcommand{\PPs}{\mathsf{P}}

\newcommand{\wgtw}{\mathtt{wgt}}
\newcommand{\ewtw}{\mathtt{ewt}}
\newcommand{\hthree}{\hspace{0.3cm}}
\newcommand{\impl}{\ \Rightarrow\ }

\title{Pathwise-randomness and models of second-order arithmetic\thanks{Supported by NSFC-11971501 and 
the Jiangsu Provincial Advantage Fund during a workshop in Nanjing in April 2019.
We thank Patey, Xia, Hirschfeldt and Jockusch for their constructive feedback.}}
\author[1]{George Barmpalias}
\affil[1]{State Key Lab of Computer Science, Inst.\ of Software, Chinese Acad.\ of Sciences, Beijing, China\vspace{0.1cm}}
\author[2]{Wei Wang}
\affil[2]{Inst.\ of Logic \& Cognition and Dept.\ of Philosophy, Sun Yat-Sen University, Guangzhou, China}

\begin{document}
\maketitle
\begin{abstract}
A tree  is {\em pathwise-random} if
all of its paths are \ml random. We show that: (a) no weakly 2-random real computes a perfect pathwise-random tree; it follows that the class of perfect
pathwise-random trees is null, with respect to any computable measure; 
(b) there exists a positive-measure pathwise-random tree which does not compute 
any complete extension of Peano arithmetic; 
and (c) there exists a perfect pathwise-random tree which does not compute any
tree of positive measure and finite randomness deficiency. 
We then obtain models of  second-order arithmetic that separate   compactness principles below
weak K\"{o}nigs lemma, answering questions by Chong \etal\ (2019).
\end{abstract}
%\setcounter{tocdepth}{1}
%\tableofcontents

\section{Introduction}\label{hCXa7dAofL}
Algorithmic randomness 
of infinite binary sequences, {\em reals},
has found applications in computability  as well as 
the analysis of the proof-theoretic strength of measure principles 
in second-order arithmetic.
{\em Trees of random reals}
are ubiquitous in measure-theoretic constructions in
computability  theory % \citep{BDNGP,dimiller} 
and, as recently suggested by \citet{Chong.Li.ea:2019},
essential in the study of compactness in reverse mathematics. 
Our goal is to:
\begin{itemize}
\item  establish essential computational properties of the {\em pathwise-random trees} 
\item  apply these to the classification of compactness principles in  second-order arithmetic.
\end{itemize}
A pleasant surprise is that some results about
pathwise-random trees are natural extensions of  well-known facts  
concerning the interplay between randomness and computational strength.

Trees of random reals are  distinct  from the  
{\em algorithmically random trees} of  
\citep{algClosedBarmp, DKH101007,DIAMOND2012519}, which express randomness
of the underlying branching mechanism in Galton-Watson processes, rather than the paths themselves.
Indeed, by the law of large numbers:
\begin{equation}\label{dz8cUP5Ndv}
\parbox{11.5cm}{algorithmically random trees (random outcomes of computable branching processes) contain 
paths which are not algorithmically random.}
\end{equation}
 In the probabilistic framework of random sets \citep{Molchanov}, as demonstrated in
 \citep{Axonphd,axon2015}, we may obtain trees of algorithmically random paths
 as random elements of  Poisson point processes.
 Using the hit-or-miss topology on the space $\TT$ 
 of trees  and the uniform measure $\mu$ on the Cantor space $\twome$,
a computable measure $\nu$ on $\TT$ can be obtained which
positively correlates $\mu(D)$ for $D\subseteq \twome$ with the
probability that the set of paths through a tree intersects $D$.
In this way, $\nu$-random trees, viewed as closed sets of reals, avoid effectively null sets: their
paths are algorithmically random.
This approach 
%of producing trees of algorithmically random paths  
has certain limitations as we will soon see.

Variables $\sigma,\tau,\rho,\eta$  denote members of the set $\twomel$ of  finite binary strings, and variables $x,y,z$  denote reals.
Let $K(\sigma)$ denote the \pf Kolmogorov complexity of  
$\sigma$:  the length of the shortest self-delimiting (prefix-free) program that generates $\sigma$.
The (randomness) {\em deficiency} of $\sigma$ is $|\sigma|-K(\sigma)$. 
The deficiency of $x$ is the supremum of the deficiencies of its prefixes;
the {\em deficiency} of
a subset of $\twomel$ or $\twome$ is the supremum of the deficiencies of its members.  
A real $x$ is  {\em random}\footnote{equivalent to the standard notion of 
\citet{MR0223179}; different randomness notions will always be qualified.} 
(for the uniform distribution $\mu$) if it has finite deficiency: $\exists c\ \forall n\ K(x\restr_n)\geq n-c$.

A {\em tree} is a subset of $\twomel$  which is downward closed with 
respect to the prefix relation $\preceq$.  
Given trees $T, T'$ we say that {\em $T$ is a subtree of $T'$} if $T\subseteq T'$, and say that
\begin{itemize}
\item $T$  is {\em pruned} if each $\sigma\in T$ has at least one extension in $T$. 
\item $T$  is {\em perfect} if each $\sigma\in T$ has at least two prefix-incomparable extensions  in $T$. 
\end{itemize}
A {\em path} of a tree is a real with all of its prefixes in $T$;
let $[T]$ denote the set of paths through $T$ and 
say that $T$ is {\em positive} if  $\mu([T])>0$. 
The {\em width} of $E\subseteq\twomel$ is the function $n\mapsto |E\cap 2^n|$.
The width of a tree $T$ is the width of $T$ as a set of strings; similarly for the deficiency of $T$.
\begin{defi}\label{cXIfSc4wDC}
A tree  is $T$ {\em pathwise-random} if $[T]$ consists of random reals.  
%and it is {\em strongly pathwise-random} if $T$ has finite deficiency.\footnote{Note that there are 
%pathwise-random trees which are not strongly pathwise-random.} 
\end{defi}
So why are pathwise-random trees useful?
Because of the following  property (see \S\ref{x1zDvaXnv2}): 
\begin{equation}\label{3b6GbFaLJQ}
\parbox{12.5cm}{{\em Universality:} Every  perfect  
pathwise-random tree computes a perfect subtree through every positive computable tree.}
\end{equation}
The class $\TT$ of pruned trees can be viewed as a homeomorphic copy of $\twome$.\footnote{Pruned trees are canonical representations of the closed sets in the Cantor space. The space of closed subsets of the Cantor space is compact, Hausdorff, without isolated points,  and has a countable base consisting of clopen sets. By Brouwer's theorem, any such space is homeomorphic to the Cantor space.}
The subclass consisting of the perfect trees 
is null with respect to any computable measure on $\TT$, see  \S\ref{DnYwAzLsF}.
%A tree $T$ is {\em pruned} if it has no dead-ends, 
%so every {\em node} $\sigma\in T$ is a prefix of a path of $T$.
%Pruned trees are representations of closed sets of $\twome$, and their collection $\TT$
%is homeomorphic to $\twome$ with respect to the standard topologies, see \S\ref{Lhc6Zs8jLS}).
%We will see in \S\ref{DnYwAzLsF} 
%that the subclass of $\TT$ consisting of proper pathwise-random trees is computationally small, in the sense that it is not
%supported by (it is $\nu$-null for) any computable measure $\nu$ on $\TT$.
%Since effective constructions with 
%non-computable measures are challenging, we resort  to topological arguments for two out of our three
%main results about pathwise-random trees, in \S\ref{Lhc6Zs8jLS}, \S\ref{36Hjxwj7uW}. 
%The intriguing possibility of deriving the existence of 
%pathwise-random trees with certain properties as 
%random outcomes of a (necessarily) non-computable probabilistic process is discussed in \S\ref{FWznALqzcD}.
%
%We caution the reader to attend to the  distinction  between {\em pruned and non-pruned trees} in the statement of our results.
%The set of pruned trees can be viewed as a compact space, homeomorphic to $\twome$, and often play the role of oracles. 
%Non-pruned trees typically appear in the output end of computations and
%statements of the type ``oracle $z$ cannot enumerate a tree with a property $P$''.

\subsection{Our results on pathwise randomness}\label{DnYwAzLsF}
An {\em array of reals} is a sequence $(x_i)$ of distinct\footnote{pairwise non-equal} reals, and its deficiency is the deficiency of
$\sqbradn{x_i}{i\in\Nat}$. 
Any random real $x$ can be effectively decomposed into an array of random reals.
Consider the even and odd bit-sequences $x_0,x_1$ of $x$, 
so $x=x_0\oplus x_1$ where $\oplus$ is the join operator. Repeating this, 
we get $x_1=x_{10}\oplus x_{11}$,  
$x_{11}=x_{110}\oplus x_{111}$ and so on, generating: 
\[
A_x:=\sqbrad{x_{1^n\ast 0}}{n\in\Nat}
\hspace{0.5cm}\textrm{which we call  {\em the van Lambalgen array} of $x$}
\]
where $\ast$ denotes concatenation, which is 
an $x$-computable array. 
For each $n>0$, 
by \citet{vLamb90}  the real $x_{1^n\ast 0}$ is random relative to $\oplus_{i<n} x_{1^i\ast 0}$, so $A_x$
is infinite and consists of distinct random reals.
Hence every random real computes an array of random reals.
In contrast, we show that this fact is no longer true if we require a uniform bound on the randomness deficiency of
the reals in the array. 
Let  $\leq_T$ denote the Turing reducibility: $x\leq_T y$ means that $x$ is computable from $y$.
A \ml random real which forms a minimal pair with $\emptyset'$ in the Turing degrees is called
{\em weakly 2-random}.
\begin{thm}\label{G6z9r7C93e}
If $z$ is weakly 2-random:
\begin{enumerate}[\hspace{0.3cm}(a)]
\item no $z$-computable array of distinct reals has finite deficiency;  
\item there is no $z$-computable perfect pathwise-random tree.
\end{enumerate}
\end{thm}
Random reals that can compute a complete
extension of Peano arithmetic, in the form of a real encoding the theory, a {\em PA real}, 
also compute the halting problem $\emptyset'$ \citep{MR2258713frank} (also see \citep{Levin:2013:FI}), hence they are not weakly 2-random.
On the other hand, every PA real computes a perfect pathwise-random tree since
there exists a non-empty \pz class consisting of perfect pathwise-random trees. This was implicitly noted
in \citep[]{Chong.Li.ea:2019}, using the computable lower bounds on the density of \pz classes of random reals obtained in \citep{MR820784}.

By Theorem \ref{G6z9r7C93e} applied to the van Lambalgen decomposition of a real:
\begin{coro}
The van Lambalgen array of any weakly 2-random real
has infinite  deficiency.
\end{coro}
Theorem \ref{G6z9r7C93e} implies a result by \citet{deniscarlsch21}
about reals that are random relative to $\emptyset'$, known as {\em 2-random reals}:
no 2-random  computes a perfect pathwise-random tree.

Theorem \ref{G6z9r7C93e} can also be viewed  as a 
strong version of \eqref{dz8cUP5Ndv}, regarding  
trees produced by computable branching processes.
Algorithmically random trees may be identified with their real codes, which are
\ml random. Although
the halting problem  computes perfect pathwise-random trees, 
``properly random'' reals or trees, already from the level of weak 2-randomness,  do not compute perfect pathwise-random trees.

The standard topology on the pruned trees  is 
generated by the finite trees $F$ whose maximal nodes have the same length $\ell$ and 
represent the class of infinite pruned trees whose restriction to $2^{\leq \ell}$ equals $F$. This  coincides with the 
hit-or-miss and  Vietoris topologies  on the closed sets of $\twome$,  see \citep[Appendix B]{Molchanov}. 
Identifying pruned trees with  closed sets of $\twome$ and using the properties oft Turing functionals as measure transformations
%\footnote{see \citep[Corollary 6.12.8]{rodenisbook}, \citep[Proposition 3.2 (i)]{BienvenuP16} and  \citep{Porter6, Bienvenu.Porter:12}} 
we get:

\begin{coro}\label{Z9i8aeyowp}
The class of perfect pathwise-random trees with respect to the standard topology on the space $\TT$ of pruned trees is
$\nu$-null, for every computable measure $\nu$ on $\TT$.
The same is true of the subclass of trees of infinite width and finite deficiency.
\end{coro}\begin{proof}
Under a canonical effective embedding of $\TT$ into $\twome$, we may view $\TT$ as a set of reals and $\nu$ as a continuous measure
on $\twome$.  
By Theorem \ref{G6z9r7C93e} (a), the upward $\leq_T$-closure of $\TT$ is $\mu$-null. 
By \citep[Proposition 3.2 (i)]{BienvenuP16},  it follows that $\TT$ is $\nu$-null. A similar argument shows that the class of 
trees of infinite width and finite deficiency is $\nu$-null for every continuous computable measure on $\TT$.
\end{proof}
There is a connection between the above results and the work of \citet[Theorem 7.7]{BienvenuP16}
on {\em deep} \pz classes, namely effectively closed sets whose members are hard to produce probabilistically. 
We say that a tree is {\em computably-growing}  if its width is bounded below by a
computable unbounded nondecreasing function.
In \citep[Theorem 7.7]{BienvenuP16} it is shown that
no incomplete \ml random real computes
any member of a computably-growing tree of finite deficiency.
The main difference  is that  Theorem \ref{G6z9r7C93e} 
applies to  trees with infinitely many paths while
\citep[Theorem 7.7]{BienvenuP16}  is restricted to
computably-growing trees.
\begin{defi}
A set $V\subseteq\twomel$ is {\em super-fat} if $2^n=\bigo{|V\cap 2^n|}$.
\end{defi}
By forcing with positive trees, in \S\ref{Lhc6Zs8jLS} we show:
\begin{thm}\label{uZbmsMF7A6}
Given $z$ and positive tree $P$ which has no $z$-\ce super-fat subset, 
every positive tree has a perfect subtree $G$ such that 
$P$ has no $(z \oplus G)$-\ce super-fat subset.
\end{thm}
Applying the argument of \S\ref{Lhc6Zs8jLS} uniformly with  each $P_c:=\sqbrad{\sigma}{\forall\rho\preceq\sigma, K(\rho)\geq |\rho|-c}$ we get:
\begin{coro}\label{OYzd34gN1g}
There exists a perfect tree of finite deficiency which cannot 
computably enumerate any super-fat set of finite deficiency; in particular, it cannot compute
any positive tree of  finite deficiency.
\end{coro}
Patey has shown that every positive tree contains a perfect subtree which does not compute any PA degree.
The following non-trivial  extension of Patey's theorem is 
motivated by a fine separation of tree-like compactness principles in \S\ref{SY9RhzaFv}.
\begin{thm}\label{xJFhzQCXsu}
There exists a positive perfect 
pathwise-random tree  which does not compute any complete extension of Peano Arithmetic.
In fact, given
any non-computable  $z$,  every
positive tree has a positive perfect subtree $T\not\geq_T z$  
which does not compute any PA real.
\end{thm}
After conference presentations of
Theorem \ref{xJFhzQCXsu}, \citep{confwuhwei,confjanjbar},  
\citet{pamiller} announced their independent discovery of the first part of Theorem \ref{xJFhzQCXsu}, in a study of covering properties.

The proofs in \S 3,  \S 4 are forcing arguments, which can be seen as adaptations of Mathias forcing to the space of trees.
Related forcing notions have been used in the study of Ramsey-type principles in reverse mathematics \citep{wang2013rainbow, dzhafarov2009ramseys}.
We do not assume prior knowledge of these arguments, but the expert in the above topic may find this existing framework useful.

\subsection{Our results on calibrating compactness principles}\label{SY9RhzaFv}
Two compactness principles are central in the development of mathematics inside
systems of second-order arithmetic: the
{\em Weak K\"{o}nig's Lemma} 
and the {\em Weak weak K\"{o}nig's Lemma}, denoted by
$\WKL$, $\WWKL$ respectively in Table \ref{iHCIFQhir}.
%%
%\begin{enumerate}[\hspace{0.5cm}(a)]
%\item {\em Weak K\"{o}nig's Lemma}: every infinite binary tree has an (infinite) path;
%\item {\em Weak weak K\"{o}nig's Lemma}: every infinite positive binary tree has an (infinite) path;
%\end{enumerate}
%%
By {\em tree} we mean binary tree, as we defined it earlier in \S\ref{hCXa7dAofL}, and by {\em positive tree} $T$
we mean that the set of its paths has positive measure. Following \citep{{Chong.Li.ea:2019}}, the latter is formally expressed by
$\exists c\ \forall n\ |T\cap 2^{n}|\geq 2^{n-c}$.
We are interested in the strength of the following  weak variations of the perfect set theorem:\footnote{Clause
(c) is not a consequence of the perfect set theorem, but rather its proof.}
\begin{enumerate}[\hspace{0.5cm}(a)]
\item  $\PPp$: {\em Every positive  tree has a positive perfect subtree} 
\item  $\PPs$: {\em  Every positive  tree has a perfect subtree} 
\item $\PPm$: {\em  Every positive  tree has an infinite countable family of paths}.\footnote{Formally, 
if  $T$ is a positive tree, there exists a sequence $(x_n)$ of distinct reals, all lying on $T$.}
\end{enumerate}
The following implications amongst the principles of
Table \ref{iHCIFQhir} are provable in $\RCA_0$:
\[
\WKL\to \PPp\to\PPs\to\PPm\to\WWKL
\]
where the first one was observed by
\citet{Chong.Li.ea:2019} and the others are straightforward. 
We then show that these lie strictly between $\WKL, \WWKL$ in proof-theoretic strength,
and $\PPp$ is  strictly stronger than $\PPs$.
This  answers \citep[Questions 4.1, 4.2]{Chong.Li.ea:2019},
gives a fine calibration of strengths of compactness
by means of natural tree-like principles, and provides  additional negative evidence over the belief that
most well-known theorems in formal analysis are equivalent to one of a small collection of axioms
over recursive arithmetic. 
The separations are obtained via
the construction of appropriate $\omega$-models of second-order arithmetic:
models of the form $(\omega, S, +, \cdot, 0,1,\leq)$
where $\omega$ denotes the natural numbers and $S$ is a class of subsets of $\omega$.
The basis system $\RCA$ stands for {\em recursive comprehension} 
and is a strong formalization of computable mathematics: its minimal $\omega$-model is the computable sets.
\begin{table}
\renewcommand{\arraystretch}{1.2}
\colorbox{black!5}{\begin{tabular}{rl}%{clK{3cm}}
\cmidrule[0.5pt]{1-2}
{\small $\WKL$ :}  & {\small Every infinite  tree has a path} \\
{\small $\PPp$ :}  & {\small Every positive  tree has a positive perfect subtree} \\
{\small $\PPs$ :}  & {\small Every positive  tree has a perfect subtree} \\
{\small $\PPm$ :}  & {\small Every positive  tree has an infinite countable family of paths} \\
{\small $\WWKL$ :}  & {\small Every positive  tree has a path} \\
\cmidrule[0.5pt]{1-2}
\end{tabular}}\centering
\caption{Compactness principles 
derived by weakening  weak K\"{o}nig's lemma}\label{iHCIFQhir}
%\vspace{-0.5cm}
\end{table}
\begin{thm}[Models]\label{rms6rS57S4}
Each of the following extensions of $\RCA$ has an $\omega$-model:
\begin{enumerate}[(a)]%\hspace{0.3cm}
\item $\WWKL+\neg \PPm$: every positive  tree has a path but  
some positive tree does not have an infinite countable family of paths.
\item  $\PPs+\neg \PPp$: every positive  tree has a perfect subtree but some positive  tree has no positive perfect subtree.
\item $\PPp+\neg \WKL$: every positive  tree has a positive perfect subtree but some infinite  tree has no path.
\end{enumerate}
\end{thm}
There are certain flexibilities regarding the choice of these models, for example cone-avoidance: they can
avoid containing any given non-zero Turing degree. These properties follow from the results of \S\ref{DnYwAzLsF}
and their refinements in \S\ref{CuGmiIbvI}, \S\ref{Lhc6Zs8jLS}, \S\ref{36Hjxwj7uW}. 
The overlap 
%between Theorem \ref{rms6rS57S4} and
with the recent work of \citet{deniscarlsch21} and \citet{pamiller} is clarified in \S\ref{Ug2xqnGYTz}.

\subsection{Background and recent work on models of compactness}\label{Ug2xqnGYTz}
Compactness is a generalization of finiteness \citep{amathmoncomp} and 
has played a leading role in the foundations of modern mathematics in the 19th century,
especially analysis and measure theory; today,
it continues to be crucial  across 
mathematics \citep{amathmon122.7}. This  is reflected in the formalization of mathematics
in second-order arithmetic, where
Weak K\"{o}nig's Lemma can support the formal development of
a large fragment of analysis,
on the basis of a standard system of computable mathematics
\citep[\S IV]{Simpson:2009.SOSOA}.
% including 
%the separable Hahn/Banach Theorem,
%Brower's fixed point theorem, 
%Heine/Borel theorem for compact metric spaces,
%Weierstrauss approximation and
%the Cauchy/Peano theorem for ordinary differential equations, see
%\cite[\S IV]{Simpson:2009.SOSOA}. 
Weak weak K\"{o}nig's Lemma is weaker  but sufficiently strong for
the formal development of much of measure theory \citep{yuxiaokang_phd,Yu.Simpson:1990}.
We focus on the principles between $\WKL$ and $\WWKL$, displayed in Table \ref{iHCIFQhir}.
Background on reverse mathematics can be found in  \citep{Simpson:2009.SOSOA, denislicing}, while 
\citep{AVIGAD20121854,Brown_vitali, KjosHanssenMStrans, binnskjos2009} focus on
formal  principles below $\WKL$. 

{\bf Models of compactness.}
 Formal implications of principles are always considered over the basic
system  $\RCA_0$ which is recursive arithmetic with $\Sigma^0_1$ induction.\footnote{The subscript 0 indicates
the presence of $\Sigma^0_1$ induction, and in most texts the reader will find the notation $\WKL_0$ in place of $\WKL$, emphasizing
the strength of induction in $\RCA_0+\WKL$. We do not need to be so specific here as our separations do not depend on the amount of
induction available.}
From a foundational point of view, the system $\RCA_0+\WKL$ is
of special interest as it can be seen as an expression of  primitive recursive arithmetic 
(Hilbert's finitism): it includes and is conservative over 
primitive recursive arithmetic with respect to $\Pi^0_2$ 
sentences.\footnote{it does not prove any such sentences that are not provable in primitive recursive arithmetic.}
%In this sense,  $\WKL$ is mathematically powerful and logically weak.
%
Separation of a pair of principles is  obtained by constructing
a model of one of them, in which the other principle is not true.
We restrict our attention to {\em $\omega$-models} of second-order arithmetic, 
which are models whose first-order part is standard: they are of the form $(\omega, S, +, \cdot, 0,1,\leq)$
where $\omega$ denotes the natural numbers and $S$ is a class of subsets of $\omega$.
An $\omega$-model $(\omega, S, +, \cdot, 0,1,\leq)$ 
is strictly determined by its {\em universe} $S$, so for simplicity we may identify the two.
The basis system $\RCA_0$ is a formal expression of computable mathematics, and it has a minimal $\omega$-model: the
class of computable sets.
The $\omega$-models $S$ of the basis system are  the Turing ideals: the
classes of subsets of $\omega$ that are closed under the join operator in the Turing degrees and 
downward closed under the Turing reducibility. 
%%
%Our results of \S\ref{SY9RhzaFv}
%show a proper hierarchy of principles between 
%$\WKL$ and $\WWKL$.
%The  separation between $\WKL$ and  $\WWKL$ was originally obtained by a random model in
%\citep{Yu.Simpson:1990}, while   
%\citep{MR2258713frank}  demonstrated that such $\omega$-models can be generated more flexibly by incomplete \ml random reals,
%since such reals do not compute $\PA$ reals.
%{\bf Recent work between $\WKL$ and $\WWKL$.}
%\citet{Chong.Li.ea:2019} started an investigation of measure-theoretic versions of the perfect set theorem
%and their line of research gained traction almost instantly. 
%\citet{pamiller} studied  the discrete and continuous covering properties and their relationship with PA degrees and $\WKL$.
%Through their covering properties they have obtained 
% the first part of Theorem \ref{xJFhzQCXsu} and
%clause (b) of Theorem \ref{rms6rS57S4}, independently.
%\citet{luliumajma} extended the work on covering properties showing, amongst other things, that the
%discrete covering property does not imply the continuous covering property.
%\citet{deniscarlsch21} investigate an array of tree-like principles
%of similar but often distinct proof-theoretic strength, towards an analysis of  a theorem of
%\citet{mycielskimea} in reverse mathematics.

\subsection{Algorithmic randomness and trees}\label{x1zDvaXnv2}
The full binary tree represents the Cantor space, with topology generated by the sets
\[
\dbra{\sigma}:=\sqbrad{z\in\twome}{\sigma\prec z}
\hspace{0.5cm}\textrm{and}\hspace{0.5cm}
\dbra{V}=\cup_{\sigma\in V}\dbra{\sigma}
\hspace{0.5cm}\textrm{for $V\subseteq\twomel$}
\]
and the uniform measure
given by $\mu(\sigma):=\mu(\dbra{\sigma})=2^{-|\sigma|}$ and
$\mu(V):=\mu(\dbra{V})$.
\ml randomness can be defined with respect to any computable measure on $\twome$, but
when no such qualification is determined, the uniform probability measure is assumed.
A \ml test is a uniformly \ce sequence of $V_i\subseteq\twomel$ such that $\mu(V_i)<2^{-i}$.
Following \citet{MR0223179}, we say that $x$ is random if $x\not\in \cap_i \dbra{V_i}$ for all such tests $(V_i)$.
This notion can be relativized to any oracle; random reals are also referred to as {\em 1-random}; 
 {\em 2-random} means random relative to the halting problem $\emptyset'$. 
A $\Pi^0_2$ class of reals is a set of the form $\cap_i \dbra{V_i}$ where $V_i$ are uniformly \ce sets of strings.
We say that $x$ is {\em weakly 2-random} if it is not a member of any  $\Pi^0_2$ class of measure 0. 
The weakly 2-random reals are exactly the \ml random reals which form a minimal pair with $\emptyset'$ in the Turing degrees, see
\citep[Corollary 7.2.12]{rodenisbook}.

Equivalently, randomness can be defined in terms of incompressibility, via Kolmogorov complexity.
A \pf machine is a Turing machine whose domain is a \pf set of strings. 
The \pf Kolmogorov complexity of $\sigma$ with respect to a \pf machine $M$, denoted by $K_M(\sigma)$,
is the length of the shortest input $\rho$ such that $M(\rho)$ converges and outputs $\sigma$. There are optimal
\pf machines $M$, such that $\sigma\mapsto K_M(\sigma)$ is minimal up to a constant, with respect to all \pf machines.
We fix an optimal \pf machine and let $K(\sigma)$ denote the \pf complexity of $\sigma$ with respect to this fixed machine.
Recall the notion of (randomness) {\em deficiency} given earlier, above Definition \ref{cXIfSc4wDC}.
A real is \ml random if and only if it has finite deficiency \citep[Theorem 6.2.3]{rodenisbook}.

The classes: 
 \[
 P_c=\{x\ |\ \forall n\ K(x\restr_n)\geq n-c\}
 \hspace{0.5cm}\textrm{and}\hspace{0.3cm}
 \hat P_c=\{\sigma\ |\ \forall \rho\preceq\sigma \ K(\rho)\geq |\rho|-c\}
 \]
are uniformly \pz and $\forall c: \mu(P_c)\geq 1-2^{-c}$ (see \citep[\S 6.2]{rodenisbook}). 
So the sets $V_i=\twomel-\hat P_i$ form a \ml test, which is {\em universal} in the sense that
$\cap_i \dbra{V_i}$ contains  exactly the set of reals which are not \ml random. By relativizing with respect to an oracle $z$, 
we may talk of the {\em $z$-deficiency}, {\em $z$-randomness} and
{\em pathwise-$z$-random} trees.

Given a tree $T$, $\sigma\in \twomel$, the tree $\sqbrad{\tau}{\sigma\ast \tau\in T}$
is called a {\em tail} of $T$. For each $F\subseteq\twomel, W\subseteq\twome$ define
$\sigma\ast F:=\sqbrad{\sigma\ast\tau}{\tau\in F}$ and 
$\sigma\ast W:=\sqbrad{\sigma\ast z}{z\in W}$.

\citet{MR820784} proved the following ergodicity property for each real $z$:
\begin{equation}\label{Yx3uuoP6eA}
\parbox{12cm}{if $F\subseteq\twome$ is   
 $\Pi^0_1(z)$  and $\mu(F)>0$, every $z$-random  $x$ has a tail in $F$:
 $\exists \sigma,\ x\in \sigma\ast F$}
\end{equation}
which shows that the $(P_i)$ are universal in a certain way,\footnote{The universality of $P_c$
in the positive \pz classes is analogous to that of the 
halting problem with respect to computable enumerations, and that of the \pz class of $PA$-complete reals with respect to \pz classes; see \cite{simpsonmassurv}.}
useful in inductive constructions of $\omega$-models of $\WWKL$.
We give a tree-analogue of \eqref{Yx3uuoP6eA} which
will be similarly useful in the construction of the models of
Theorem \ref{rms6rS57S4} and was also used by
 \citet{deniscarlsch21}. 
\begin{lem}\label{nqjFhRoRlv}
Given $z$, let $T$ be
a tree such that $[T]$ consists of $z$-random  reals and let 
$W$ be a $z$-computable positive tree. 
There exists $\sigma$ such that $[T]\cap\dbra{\sigma}\neq\emptyset$ and $[T]\cap\dbra{\sigma}\subseteq \sigma\ast [W]$. 
\end{lem}
\begin{proof}
If the claim was not true, then for each  $\sigma$:
\begin{equation}\label{8SVMYAalVq}
[T]\cap\dbra{\sigma}\neq\emptyset
\hspace{0.5cm}\Rightarrow\hspace{0.5cm}
[T]\cap\dbra{\sigma}-\sigma\ast [W]\neq\emptyset
\end{equation}
Let $Q\subseteq\twomel$ be a $z$-\ce \pf set such that $[W]=\twome-\dbra{Q}$.
Then applying \eqref{8SVMYAalVq} for $\sigma:=\lambda$, there exists
$\tau_0\in Q$ such that $[T]\cap \dbra{\tau_0}\neq\emptyset$.
Repeating the application of \eqref{8SVMYAalVq} for 
$\sigma:=\tau_0$, there exists $\tau_1\in Q$ such that 
$[T]\cap \dbra{\tau_0\ast\tau_1}\neq\emptyset$.
Continuing in this way we obtain  $(\tau_i)$ and $x:=\tau_0\ast\tau_1\ast\cdots$ such that $x\in [T]$
and $x\in \dbra{Q^n}$ for each $n$, where $Q^n$ denotes the $n$-repetition of the concatenation $Q\ast\cdots \ast Q$, 
and $Q\ast Q:=\sqbrad{\sigma\ast\rho}{\sigma,\rho\in Q}$. Since  $W$ positive, $\mu(Q)< 1$.
Since $\mu(Q^n)=(\mu(Q))^n$ and $Q$ is $z$-c.e., it follows that $x$ is not $z$-random, 
contradicting the assumption about
$T$ and concluding the proof.
\end{proof}
Universality property  \eqref{3b6GbFaLJQ} is a special case of the following.
\begin{coro}\label{nqjFhRoRlva}
Every  perfect  
pathwise-$z$-random tree computes a perfect subtree through every positive $z$-computable tree.
\end{coro}\begin{proof}
This follows by Lemma \ref{nqjFhRoRlv} since 
whenever $T$ is  perfect, the same holds for 
its restriction to any $T$-extendible node $\sigma$.
\end{proof}
{\bf A word on notation.}
We maintain some consistency in our notation, reserving 
$i, j, k, n, m, t, s, r$ for non-negative integers, $\sigma,\tau,\rho$
for strings, $x,y,z$ for reals, $f, g, h$ for functions on  integers or strings, $\Phi, \Psi$ for functionals, 
and $T, F, Q, H, W, B, E, P, D, G, S, U, V$ for sets of reals or strings, often finite or infinite trees. 
We occasionally use the accents $d', d''$ of a variable $d$ as a 
distinct variable, and note that this does not denote the jump operation, except for 
$\emptyset'$.  We write $f(n)\de, f(n)\un$ to denote that $f$ is defined or undefined
respectively, on $n$.

\section{Random reals and  pathwise-random trees: Theorem \ref{G6z9r7C93e}}\label{CuGmiIbvI}
We show that weakly 2-random reals cannot compute arrays of reals of finite deficiency.

\subsection{Hypergraphs from computations of arrays}\label{rz6t4EOhsN}
Computations of arrays of reals are induced by effective maps from strings to finite sets of strings, which in turn are  conveniently viewed as
hypergraphs whose vertices are strings and hyperedges are \pf sets of strings. 
\begin{defi}\label{O1f2ngxwtX}
A tuple $(\VV,\FF)$ refers to a hypergraph with vertex-set a \pf  $\VV\subseteq\twomel$ 
and edge-set $\FF$ consisting of finite subsets of $\VV$. A function $E\mapsto\wgt{E}$ gives a real positive weight to each edge
such that the total edge-weight (the sum of the weights of all edges) is $\leq 1$.
 The  {\em edge-weight} of $\sigma\in \VV$ is given by: 
\[
\ewt{\sigma}=\sum_{\FF\ni E\ni \sigma} \wgt{E}.
\]
We often abbreviate $(\VV,\FF)$ to $\FF$ and say that $\FF$ is $k$-fat if $|E|\geq k$ for each $E\in \FF$.
\end{defi}
We use a counting argument for the existence of {\em light} vertices in such hypergraphs, 
in the sense that their weight is considerably 
smaller than the weight of the edges they belong to.

\begin{lem}\label{bWqFoSOZ8}
Given a $k$-fat hypergraph $(\VV,\FF)$ with total edge-weight  $\delta$, 
there exists $\tau\in\VV$
such that  $k \delta\cdot 2^{-|\tau|} \leq\ewt{\tau}$.
\end{lem}\begin{proof}
Given $(\VV,\FF)$ as in the statement:
\begin{equation}\label{eNVceNvVVI}
\sum_{\sigma\in \VV} \ewt{\sigma}\geq k \delta
\end{equation}
since in this sum the weight of each edge of $\FF$ appears
at least $k$ times -- once for every vertex that it contains.
Assume, for a contradiction,
that for each vertex $\sigma\in \VV$, 
 $\ewt{\sigma}< k\delta\cdot  2^{-|\sigma|}$. 
 By \eqref{eNVceNvVVI}, since $\VV$ is \pfn:
 \[
k \delta   \leq 
\sum_{\sigma\in \VV} \ewt{\sigma}
<   k \delta \cdot \sum_{\sigma\in \VV}  2^{-|\sigma|} \leq k\delta
\]
which is the required contradiction.
\end{proof}
We now consider computations of arrays of reals and define their associated hypergraphs.
\begin{defi}\label{def:array-functional}\label{2cBbey6FMZ}
An \emph{array-functional} is a Turing functional $\Phi$
such that for each $x$, $i, j$:
\[
\parb{\Phi(x; i,j)\de\wedga \Phi(x; i,j+1)\de}\impl 
\parb{\Phi(x; i,j)\in 2^j \wedga \Phi(x; i,j)\prec \Phi(x; i,j+1)}.
\]
We let $\Phi_s$ denote the approximation of $\Phi$ at stage $s$ and $\Phi(x; i):=\cup_j \Phi(x; i,j)$.
\end{defi}

\subsection{Array-functionals and pathwise-randomness}\label{SMCqyboIt}
The  hypergraphs considered above are used in the proof of the following technical fact.
Given array functional $\Phi$, a computation $\Phi_s(x; i)\de$ with oracle-use $\ell$ may be denoted as
$\Phi_s(\eta; i)\de$ for any string $\sigma$ with prefix $x\restr_{\ell}$. By  convention, the length of all 
strings in $\Phi_s(\eta; i)\de$ is $<s$.

\begin{lem}\label{nSHH8pXU}
Given an array-functional $\Phi$,
there exists a Martin-L\"{o}f test $(V_n)$ such that 
\[
\forall n\ \ \parb{\textrm{$V_n$ is \pf\  and\ \ } \mu(Q_n) \leq 2^{-n}}
\]
where $x\in Q_n$ iff $\Phi(x; i), i< 2^{2n}$ are defined, distinct, and have no prefix in $V_n$.
\end{lem}\begin{proof}
Uniformly in $n$, we define an effective enumeration $V_{n}(s)$ of $V_n$,
where $V_{n}(0) = \emptyset$.

%Let $[V_{n}(s)]$ denote the upward $\preceq$-closure of $V_{n}(s)$.

Assuming that $V_{n}(s)$ is defined, for $t < s$ let
%and\ \  $F_n := \cup_s F_{n}(s)$
\begin{align*}
F_{n}(s)\       &:=\ \sqbrad{x}{\exists i< 2^{2n},\ \Phi_s(x; i)\cap V_{n}(s)\neq\emptyset}
\hspace{0.3cm}\textrm{and}\hspace{0.3cm}
E^x_{n}(s,t)\ :=\ \bigcup_{i< 2^{2n}}\Phi_s(x; i)\cap 2^t
\end{align*}
Consider the $2^{2n}$-fat hypergraph  $\mathcal{G}_{n}(s,t) := \parb{2^t, \mathcal{E}_{n}(s,t)}$ with
\begin{itemize}
\item $\EE_{n}(s,t)\ :=\ \sqbrad{E^x_{n}(s,t)}{\abs{E^x_{n}(s,t)}\geq 2^{2n}\wedga x\not\in\dbra{F_{n}(s)}}$
\item $X(E) := \sqbrad{x}{E\in \mathcal{E}_{n}(s,t)\wedga E = E_{n}^x(s,t)}$
\item edge-weight  $\wgtw_{n}(E, s,t) := \mu(X(E))$ for $E\in \EE_{n}(s,t)$.
\end{itemize}
The induced edge-weight of $\tau\in 2^t$ is denoted by $\ewtw_{n}(s,t,\tau)$. 

Since only the first $s$ bits of $x$ affect the computation $\Phi_s(\eta; i)$:
\begin{enumerate}[(i)]
\item $\mathcal{G}_{n}(s,t)$,  $\wgtw_{n}(E, s,t)$ and $\ewtw_{n}(s,t,\tau)$ are computable from $s,t, n,\tau$
\item $X(E), E\in \EE_{n}(s,t)$ are  pairwise disjoint, clopen and can be viewed as subsets of $2^s$.
\end{enumerate}
By (ii), the total edge-weight 
\[
g_{n}(s,t) := \sum_{E \in \mathcal{E}_{n}(s,t)} \wgtw_{n}(E,s,t)
\]
is at most 1.
Define $V_{n}(s+1)$ inductively: 
\begin{enumerate}[\hthree (a)]
\item if $\forall t<s,\ g_{n}(s,t) \leq 2^{-n}$ let $V_{n}(s+1) = V_{n}(s)$;
\item otherwise pick the least $t < s, \tau_{n}(s) \in 2^t$ with $g_{n}(s,t) > 2^{-n}$ and
\[
\ewtw_{n}(s,t, \tau_{n}(s)) \geq 2^{2n - t}\cdot g_{n}(s,t) > 2^{n-t} = 2^{n - |\tau_{n}(s)|}
\]
which exist by Lemma \ref{bWqFoSOZ8}. Let $V_{n}(s+1) := V_{n}(s) \cup \{\tau_{n}(s)\}$.
\end{enumerate}
This completes the computable definition of $V_n(s+1)$ and
the inductive definition $(V_n(s))$. 

The string $\tau_{n}(s)$ of clause (b) has $\ewtw_{n}(s,t, \tau_{n}(s))>0$, so 
\begin{itemize}
\item $\tau_{n}(s)$ belongs to an edge of $\mathcal{G}_{n}(s,t)$
\item by the definition of  the edge-set of $\mathcal{G}_{n}(s,t)$, 
$\tau_{n}(s)$ has no prefix in $V_{n}(s)$.
\end{itemize}
By induction, $V_n := \cup_s V_{n}(s)$ is prefix-free.
Since $(V_n(s))$ is computable, the $V_n$ are uniformly c.e., and it remains to establish 
$\mu(V_n) < 2^{-n}$. Note that
\[
V_{n}(s+1) \neq V_{n}(s)\impl 
  \mu(\dbra{V_{n}(s+1)} - \dbra{V_{n}(s)}) = 2^{-|\tau_{n}(s)|} \leq 2^{-n}\cdot \ewtw_{n}(s,t, \tau_{n}(s)),
\]
where $t = |\tau_{n}(s)|$ and $\ewtw_{n}(s,t, \tau_{n}(s))=\mu(\{\eta\in 2^{s} : \dbra{\eta}\cap \dbra{F_{n}(s)}=\emptyset\wedga 
\tau_{n}(s) \in E_{n}^\eta(s,t)\})$. 

Since  $\tau_{n}(s) \in E_{n}^{\eta}(s,t)\impl \dbra{\eta}\subseteq \dbra{F_{n}(s+1)}$ we have
\[
\ewtw_{n}(s,t, \tau_{n}(s)) \leq \mu(F_{n}(s+1)-F_{n}(s)).
\]
so the growth of $\mu(V_n)$  is dominated by the growth of $\mu(F_n)$. So
$\mu(V_n) < 2^{-n}\cdot  \mu(F_n) \leq 2^{-n}$.

Finally, we show that $\mu(Q_n) \leq 2^{-n}$ by way of contradiction.
Assuming that $\mu(Q_n) > 2^{-n}$,
there exists $t$ such that $g_{n}(s,t) > 2^{-n}$ for all sufficiently large $s$.
The construction would then eventually enumerate all $\tau \in 2^t$ into $V_n$, making $Q_n$ empty, which is a contradiction.
\end{proof}

\begin{lem}\label{W6K2uziWtO}
If $V_i\subseteq\twomel$ are \pf and uniformly \ce  with $\mu(V_i)< 2^{-i}$ then
\[
\exists c\ \forall \sigma\in V_{2i}: K(\sigma)\leq |\sigma|-i+c.
\]
\end{lem}\begin{proof}
By the minimality  of $K$, it suffices to produce a \pf machine $M$ with
\[
\forall i\ \forall \sigma\in V_{2i}: K_{M}(\sigma)\leq |\sigma|-i.
\]
The weight of the domain of such $M$ is
\[
\sum_i 2^i\cdot \mu(V_i)< \sum_i  2^{-i}=1.
\]
so by the Kraft-Chaitin-Levin theorem \citep[Theorem 3.6.1]{rodenisbook} such a machine exists. 
\end{proof}
\begin{coron}[Theorem \ref{G6z9r7C93e}]
If $z$ is weakly 2-random then 
\begin{enumerate}[\hspace{0.3cm}(a)]
\item no  $z$-computable array of reals has  finite deficiency
\item there is no $z$-computable perfect pathwise-random tree.
\end{enumerate}
\end{coron}\begin{proof}
We first show how to derive (b) from (a), by proving the contrapositive. 

Consider the \pz positive trees $B:=\sqbrad{\sigma}{\forall \rho\preceq\sigma,\ K(\rho)\geq |\rho|-1}$.

Let $W$ be a computable tree with $[W]=[B]$ and $B\subseteq W$. 
By Lemma \ref{nqjFhRoRlv}  and  $\neg$(b), $z$
computes a perfect subtree $T$ of $W$. So $[T]\subseteq [W]\subseteq [B]$ and
since $T$ does not have dead-ends,
it computes an array of pairwise distinct reals in $[B]$. 
By the choice of $B$, this is an array of reals of finite deficiency and is computable by $z$. So $\neg$(a).

For (a), suppose that $z$ computes an array $(y_i)$ of reals with  finite deficiency. 
The $y_i$ are pairwise non-equal 
and there exists an array functional $\Phi$ which is total on $z$ and the output $\Phi(z)$ is the array$(y_i)$.
It suffices to show that $z$ is not weakly 2-random. 
Let 
\begin{itemize}
\item $[\Phi(z)]:=\sqbrad{y_i}{i\in\Nat}$. 
\item $(V_i)$, $(Q_n)$  be as in Lemma \ref{nSHH8pXU} with respect to $\Phi$.
\end{itemize}
By Lemma \ref{W6K2uziWtO} and since $[\Phi(z)]$ has finite deficiency, 
$\exists i_0\ \forall i\geq i_0:\ [\Phi(z)]\subseteq \twome -\dbra{V_i}$.

By Lemma \ref{nSHH8pXU} it follows that $z\in \cap_{i\geq i_0} Q_i$ and $\mu(\cap_{i\geq i_0} Q_i)=0$. 

Since $\cap_{i\geq i_0} Q_i$ is a $\Pi^0_2$ class, it follows that $z$ is not weakly 2-random. 
\end{proof}
We  now construct an $\omega$-model for $\WWKL+\neg P^{-}$,
proving clause (a) of Theorem \ref{rms6rS57S4}.
Let $P$ be a non-empty \pz class of reals
with a uniform bound on their randomness deficiency, and let $\hat{P}$ be a computable tree such that $[\hat{P}]=P$.
Let $x$ be a weakly 2-random real, consider the van Lambalgen array $T_x=\sqbrad{x_{1^i\ast 0}}{i\in \Nat}$
as this was defined in \S\ref{DnYwAzLsF},  let $z_n:=\oplus_{i\leq n}x_{1^i\ast 0}$ and let $\II$ be the model which is the union of all $\II_n:=\sqbrad{z}{z\leq_T z_n}$.
By Theorem \ref{G6z9r7C93e}, no $x\in \II$ computes a proper subtree of $\hat{P}$, and $\hat{P}\in \II_0\subseteq \II$.
Given positive tree $T\in\II$ there exists $n$ such that $T \in\II_n$ so by van Lambalgen's theorem, $z_{n+1}$ is random relative to $T$. 
By \eqref{Yx3uuoP6eA}, 
$z_{n+1}$  computes a path through $T$, so some path through $T$ is definable in $\II$.

Let $n$-$\WWKL$ be the axiom that every  $\emptyset^{(n)}$-computable positive tree has a path.
Noted that if  $n\geq 2$ and we start building the model using  an $n$-random real, a similar argument produces a model of 
$n$-$\WWKL +  \neg\PPm$.

\section{Avoiding  positive trees: proof of Theorem \ref{uZbmsMF7A6}}\label{Lhc6Zs8jLS} 
Let $\TT$ be the class of pruned infinite trees and $\TTast:=\cup_n \sqbrad{2^n\cap T}{T\in\TT}$, which we view as a set of finite trees 
(the trees obtained by taking the $\preceq$-downward closure).

We prove Theorem \ref{uZbmsMF7A6} (and, by the same argument, Corollary \ref{OYzd34gN1g}):
\begin{equation*}
\parbox{13cm}{Given $z$ and positive tree $P$ which has no $z$-\ce super-fat subset, 
every positive tree has a perfect subtree $G$ such that 
$P$ has no $(z \oplus G)$-\ce super-fat subset.}
\end{equation*}
Let $\mu_{\tau}$ denote the uniform measure relative to $\dbra{\tau}$: 
$\mu_{\tau}(E):= 2^{|\tau|}\cdot \mu(E\cap\dbra{\tau})$ for  measurable $E\subseteq\twome$.
Given $T\in\TT$, $F, F'\in\TTast$ we say that:
\begin{itemize}
\item {\em $\sigma$ is extendible in $T$} if $[T]\cap\dbra{\sigma}\neq\emptyset$
\item {\em $F$ is extendible in $T$} if each $\sigma\in F$ is extendible in $T$.
\item $F\preceq F'$ (prefix relation) if $F = \{\sigma \in 2^n: \sigma \text{ is a prefix of some } \tau \in F'\}$ for some $n$; 
 if in addition $F\neq F'$, we write $F\prec F'$.
\end{itemize}
We say that  {\em $F$  is a prefix of $T$} (or that $T$ is an extension of $F$), denoted by $F\prec T$, 
if $F = T \cap 2^n$, where $n$ is the height of $F$. 

\subsection{Forcing conditions and Turing functionals}
Given $\delta>0$ we say that $T\in\TT$ is a {\em $\delta$-robust extension} of $F \in\TTast$, denoted by  
$F\prec_\delta T$, if $F\prec T$ and 
$\forall \sigma\in F,\ \mu_{\sigma}([T])> \delta$.
We say that $T$ is a {\em robust extension} of $F$ if $\exists \delta>0,\ F\prec_{\delta} T$.

\begin{defi}\label{rgE7gtKH8R}
A {\em condition} is a pair $(F, T)$ where
$F\in\TTast$, $T\in \TT$ and $T$ is a robust extension of $F$.
If $F \prec_\delta T$ we say that $(F, T)$ is \emph{$\delta$-robust}.
We say that   
$S\in\TT$ \emph{meets} condition $(F,T)$ if $F \prec S \subseteq T$, and
 identify $(F,T)$ with the set of $S\in\TT$ that meet  $(F,T)$. 
We say that  $(F', T')$ \emph{extends}  $(F, T)$, and write $(F', T') \leq (F, T)$, if
$F \preceq F'$, $T' \subseteq T$, and every member of $F$ has two $\preceq$-incomparable
extensions in $F'$.
\end{defi}
The extension relation is a partial order on the set of conditions.
Any sufficiently generic chain of conditions defines a  perfect tree.
By the Lebesgue density theorem:
\begin{equation}\label{5gEofWIwIJ}
\textrm{$\forall \epsilon>0$\ \  every condition $(F, T)$ has a $(1-\epsilon/|F'|)$-robust extension.}
\end{equation}
For every positive tree $T$, $(\{\lambda\}, T)$ is a condition. 
Let $(\Phi^z_e)$ be an effective list of Turing functionals with oracle $z$, 
accepting oracle $T\in \TT$ and integer input $k\geq 0$, where access to the oracle is via gradual queries to its prefixes.
The output of a $\Phi^z_e(T;k)$-computation is a finite subset of $2^k$ of size $2^{k-e}$. We assume that
for each $T$ and  $(T\oplus z)$-\ce super-fat $E\subseteq\twomel$: 
\begin{equation}\label{t4qGJVpf7l}
\exists e\ \forall k\ \parlr{\Phi^z_e(T;k)\de\ \subseteq E\cap 2^k\ \wedge\ |\Phi^z_e(T;k)|=2^{k-e}}.
\end{equation}
For simplicity we omit $z$ from the expressions. Let
\[
 \Phi_e^{-1}(\sigma)[s] := \sqbrad{F\in\TTast}{\sigma\in \Phi_e(F; |\sigma|)[s] \de}
\]
and $\Phi_e^{-1}(\sigma) := \bigcup_s \Phi_e^{-1}(\sigma)[s]$.

\subsection{Hitting-sets and regularity}
In order to force the existence of the required perfect tree, we use the following notions.
\begin{defi}
We say that $V \subseteq \twomel$ is a \emph{hitting-set} for  $\mathcal{F}\subseteq\TTast$
if no $F \in \FF$ has an extension $T\in\TT$ with $[T]\subseteq 2^\omega - \dbra{V}$.
The \emph{hitting-cost} of $\mathcal{F}$ is 
\[
\cost{\FF} = \inf\sqbradb{\mu(V)}{V \text{ is a hitting-set for } \FF}.
\]
The {\em hitting-cost of $\mathcal{F}$ relative to $E\in\TTast$} is:
\[
\cost{\FF,E} = \inf \sqbradB{\max_{\rho\in E} \mu_{\rho}(V)}{V \text{ is a hitting-set for } \sqbrad{F\in\FF}{E\preceq F}}.
\]
 \end{defi}
Since $\Phi^{-1}(\sigma)[s]$ is increasing in $s$, 
$\cost{\Phi^{-1}(\sigma)[s], E}$ is also increasing in $s$, and bounded above by $1$.
Hence $\lim_s \cost{\Phi^{-1}(\sigma)[s], E}$ exists.
In the proof below we use the fact that for each $H\subseteq 2^{\leq n}$ there exists $H'\subseteq 2^n$ such that
$\dbra{H}=\dbra{H'}$.

\begin{lem}[Regularity]\label{csvmfhW6}
Given $F\in\TTast$,  $\delta = \lim_s \cost{\Phi_e^{-1}(\sigma)[s], F}$, $\epsilon > 0$:
\begin{enumerate}[\hspace{0.3cm}(a)]
\item there exists a hitting-set $H \subseteq 2^{<\omega}$ for $\Phi_e^{-1}(\sigma)$  with
$\forall \rho\in F, \ \mu_\rho(H) < \delta + \epsilon$.
\item  $\cost{\Phi_e^{-1}(\sigma), F} = \lim_s \cost{\Phi_e^{-1}(\sigma)[s], F}$.
\end{enumerate}
\end{lem}

\begin{proof}
Let $\mathcal{H}$ be the set of finite sequences $(V_0, V_1, \ldots, V_{t-1})$,
such that for each $i\leq j<t$:
\begin{itemize}
 \item $V_i \subseteq 2^i$ and $V_i$ is a hitting-set for $\Phi_e^{-1}(\sigma)[i]$
 \item $\dbra{V_i} \subseteq \dbra{V_{j}}$ and $\forall \rho\in F,\ \mu_\rho(V_i) \leq \delta$
\end{itemize}
Then $\mathcal{H}$ is a finitely branching tree, since $2^s$ only has  finitely many subsets.
By the hypothesis, for each $t>0$ there exists a hitting-set $U_t \subseteq 2^t$
for $\Phi_e^{-1}(\sigma)[t]$ with $\forall \rho\in F, \mu_{\rho}(U_t)\leq \delta$.
By monotonicity of $\Phi_e^{-1}(\sigma)[s]$, the set $U_t$ is also a
a hitting-set for $\Phi_e^{-1}(\sigma)[s]$, $s < t$.
Hence
\begin{equation*}
\parbox{12cm}{for each $t$, $\mathcal{H}$ contains a path $(V_0, V_1, \ldots, V_{t-1})$ of length $t$, with $V_{t-1}=U_t$}
\end{equation*}
so $\mathcal{H}$ is an infinite finite-branching tree
and by K\"{o}nig's lemma it has an infinite path $(V_i)$. Then
$V = \cup_s V_s$ is a hitting-set for $\Phi_e^{-1}(\sigma)$ with
$\forall\rho\in F,\ \mu_\rho(V) \leq \delta$, so  (a),  (b) hold.
\end{proof}

\subsection{Forcing convergence or divergence}
We show how to build chains of conditions that determine
a perfect tree $G$ satisfying the requirements of Theorem \ref{uZbmsMF7A6}.
The goal  is to either force $\Phi_e(G)$ to output a string outside $P$ (convergence), or force it to be partial (divergence).
Given $\delta>0$ and a $\delta$-robust condition,  if
\begin{equation}\label{rqNv8183qi}
\exists \epsilon>0\ \exists k_0\ \forall k>k_0:\ \left| \{\sigma \in 2^k: \cost{\Phi_e^{-1}(\sigma), F} > 1 - \delta \} \right| > \epsilon\cdot 2^k.
\end{equation}
we can extend $(F, T)$ in a way that forces $\Phi_e$ to output a string outside $P$. Indeed, it suffices to
find an extension $E\in\TTast$ of $F$ such that $\Phi_e(E;|\sigma|)\de\not\subseteq P$ 
and $\forall \rho\in E,\ \mu_{\rho}([T])>0$. By \eqref{rqNv8183qi}
there exists a $z$-\ce super-fat set of strings $\sigma$, such that the cost of hitting
all trees that $\Phi_e$-map to extensions of $\sigma$ is large. By our hypothesis about $z$, such a set cannot be
a subset of $P$, so there exists $\sigma\not\in P$ such that the hitting-cost of $\Phi_e^{-1}(\sigma)$ is large.
This, along with the robustness of $(F,T)$, allows us to find the required extension $E$ such that 
$\sigma\in\Phi_e(E;|\sigma|)$. 

\begin{lem}[Forcing convergence]\label{BWZl9FShCb}
Let $(F,T)$ be a $\delta$-robust condition such that
\[
\liminf_k 2^{-k}\cdot \left| \{\sigma \in 2^k: \cost{\Phi_e^{-1}(\sigma), F} > 1 - \delta \} \right| > 0.
\]
There exists $(F', T')< (F,T)$ with $\exists n\ \Phi_e(F'; n)\de \not\subseteq P$.
\end{lem}\begin{proof}
Fix $k_0, \epsilon>0$ such that
\[
\forall k>k_0:\ \left| \{\sigma \in 2^k: \cost{\Phi_e^{-1}(\sigma), F} > 1 - \delta \} \right| > \epsilon\cdot 2^k
\]
so the set $D:=\sqbradlr{\sigma}{|\sigma|> k_0\ \wedge\  \cost{\Phi_e^{-1}(\sigma), F} > 1 - \delta}$ is super-fat.

Recall that we have suppressed oracle $z$ in $\Phi_e$, so $\Phi_e^{-1}(\sigma)$ is $z$-\ce and 
$\cost{\Phi_e^{-1}(\sigma), F} > 1 - \delta$ is a $\Sigma^0_1(z)$ condition.
It follows that $D$ is a $z$-\ce set of strings.

By the hypothesis about $z$ we get $D \not\subseteq P$, so fix $\sigma \in D - P$ and $s$ with  
\begin{equation}\label{eq:S3-conv}
\cost{\Phi_e^{-1}(\sigma)[s], F}) > 1 - \delta.  
\end{equation}
It remains to show that there exists $E\succeq F, E\in \Phi_e^{-1}(\sigma)[s]$ 
with  $\forall \xi\in E,\ \mu_\xi([T]) > 0$.

Assuming otherwise, for each $E\succeq F, E \in \Phi_e^{-1}(\sigma)[s]$ 
we can pick $\xi_E \in E$, $\mu_{\xi_E}([T]) = 0$ so
\[
H:=\sqbrad{\xi_E}{E \in \Phi_e^{-1}(\sigma)[s]\ \wedge\ F \preceq E}
\]
is a hitting-set for $\sqbrad{G\in \Phi_e^{-1}(\sigma)[s]}{F\preceq G}$.
By the choice of $\xi_E$ and  robustness of $(F,T)$:
\[
\forall \tau\in F, \ \mu_\tau(\dbra{H}) \leq 1 - \delta
\hspace{0.3cm}\textrm{so}\hspace{0.3cm}
 \cost{\Phi_e^{-1}(\sigma)[s], F} \leq 1 - \delta
\]
which contradicts \eqref{eq:S3-conv} and concludes the proof.
\end{proof}

Next we explain how we can extend $(F, T)$ in order to force divergence of $\Phi_e$, 
in the case that the condition of Lemma \ref{BWZl9FShCb} does not hold.
For simplicity, assume that $F=\{\lambda\}$ and let $T$ be any positive tree, so we need to find
$k$ and a positive subtree $T'$ of $T$ such that $\Phi_e(S;k)\un$ for all subtrees $S$ of $T'$.
By the hypothesis, given any $\delta\in (0,1)$ we can find $k$ such that 
\begin{equation}\label{axrZbQYJVm}
\left| \{\sigma \in 2^k: \cost{\Phi_e^{-1}(\sigma), F}) > 1 - \delta \} \right| < 2^{k-e}.
\end{equation}
The intuition here is that given $x\in [T]$, we can consider the class of subtrees $S$ of $T$ containing $x$,
and then the set of $\sigma\in 2^k$ in the outputs $\Phi_e(S;k)$ that are defined. In this sense, each $x$ corresponds to
a set of $k$-bit strings of size at least $2^{k-e}$. Here it is crucial that for each $S$ there are less than
$2^{2^k}$ possibilities for the output of $\Phi_e(S;k)$.

The idea is to restrict $T$ sufficiently, to a subtree $T'$,  so that
every $x\in [T']$ corresponds to less than $2^{k-e}$ many $k$-bit strings, which means that  $\Phi_e(S;k)\un$ for any subtree $S$ of $T'$.
In \eqref{axrZbQYJVm} we may think that $1-\delta$ is sufficiently small, so for most $\sigma\in 2^k$ there exists a small open hitting-set 
for $\Phi_e^{-1}(\sigma)$.  The set of reals that belong to sufficiently many of these small hitting-sets has small measure, which can be removed
from $T$ by removing a sufficiently tight open cover of it. 
The remaining set contains paths of a positive subtree $T'$ of $T$.
This will be done in such a way that  the subtrees $S$ of $T'$ can only $\Phi_e$-map to less than $2^{k-e}$ many
$k$-bit strings -- otherwise
$S$ would contain a path in one of the hitting-sets that we have removed.

Below is the formal details of this argument.

\begin{lem}[Forcing Divergence]\label{5XvJQvkQxh}
Suppose that $\delta \geq 1 - 2^{-n-e-4}$ and  
\begin{itemize}
\item $(F,T)$ is a  $\delta$-robust condition with $|F|<2^n$ \vspace{0.1cm}
\item $\liminf_k 2^{-k} \left| \{\sigma \in 2^k: \cost{\Phi_e^{-1}(\sigma), F}) > 1 - \delta \} \right|= 0$.
\end{itemize}
There exists $k$ and  $(F', T')< (F, T)$
 such that $\Phi_e(S; k) \uparrow$ for all trees $S\in (F', T')$.
\end{lem}\begin{proof}
By the hypothesis we may fix $k$ such that
\begin{equation}\label{eq:S3-div-k}
 |C_k| \geq 2^k-2^{k-e-1}
 \hspace{0.3cm}\textrm{where}\hspace{0.3cm}
 C_k := \sqbrad{\sigma \in 2^k}{\cost{\Phi_e^{-1}(\sigma), F}  \leq 1 - \delta}.
\end{equation}
Fix $\rho\in F$ and for each $\sigma \in 2^k$ let
\begin{itemize}
\item $H_{\sigma}$ be a hitting-set for $\Phi_e^{-1}(\sigma)$ with $\forall \sigma\in C_k\ \forall \rho\in F,\ \mu_\rho(H_{\sigma}) < 2\cdot (1-\delta)$.
\item $M_{\rho} :=\sqbrad{x \in \dbra{\rho}}{|\sqbrad{\sigma \in C_k}{x \in \dbra{H_{\sigma}}}| > 2^{k-n-e-1}}$
and $M'_{\rho} :=\dbra{\rho}-M_{\rho}$.
\end{itemize}

Each $x\in M_{\rho}$ belongs to more than $2^{k-n-e-1}$ of the  $\dbra{H_{\sigma}}, \sigma \in C_k$, so
\[
\mu_\rho(M_{\rho}) \cdot 2^{k-n-e-1}< \sum_{\sigma \in C_k} \mu_\rho(H_{\sigma}).
\]
Since $\abs{C_k}\leq 2^k$ we have 
$\sum_{\sigma \in C_k} \mu_\rho(H_{\sigma}) < 2\cdot \abs{C_k} \cdot (1 - \delta)\leq 2^{k+1} \cdot (1 - \delta)$ so
\[
2^{k+1} \cdot (1 - \delta) > \mu_\rho(M_{\rho}) \cdot 2^{k-n-e-1}
\]
and since $\delta \geq 1 - 2^{-n-e-4}$ we get $\mu_\rho(M_{\rho}) <  1/4$.
Since $(F,T)$ is a  $\delta$-robust,
\begin{equation}\label{eq:S3-div-N}
\forall \rho\in F,\ \mu_\rho(M'_{\rho}\cap [T]) > 1/2.
\end{equation}
Each $x\in M'_{\rho}$ corresponds to  the set $\{\sigma \in C_k : x \in \dbra{H_{\sigma}}\}$ which has size 
$\leq 2^{k-n-e-1}$.

If we let $\DD$ be the class of subsets of $C_k$ of size $\leq 2^{k-n-e-1}$ then
\[
M'_{\rho}=\bigcup_{D\in \DD} M'_{\rho}(D)
\hspace{0.3cm}\textrm{where}\hspace{0.3cm}
M'_{\rho}(D):= \{x \in M'_{\rho} : D = \{\sigma \in C_k : x \in \dbra{H_{\sigma}}\}\}
\]
gives a partition of $M'_{\rho}$ into $<2^{2^k}$ classes. The above apply to each $\rho\in F$, so
\[
\forall \rho\in F\ \exists D_\rho\in \DD:\ \ 
\parb{\mu_\rho(M'_{\rho}(D_{\rho})\cap [T]) > 2^{-1-2^k}\wedga |D_{\rho}|\leq 2^{k-n-e-1}}.
\]
So there exists  an infinite pruned tree $T'\subseteq T$, $T'\succ F$ which is a robust extension of $F$ and
\[
[T'] \subseteq \bigcup_{\rho\in F} M'_{\rho}(D_{\rho}).
\]
We show that if $S\in (F, T')$  then $\Phi_e(S;k)\un$.
For a contradiction,  assume otherwise. We show
\begin{equation}\label{QnXMnFzaMQ}
\sigma \in \Phi_e(S; k)\cap C_k\impl 
\exists \rho \in F,\ \sigma \in D_\rho.
\end{equation}
Since $S\subseteq T'$ and $\sigma \in \Phi_e(S; k)\cap C_k$, 
the hitting set $H_{\sigma}$ for $\Phi^{-1}(\sigma)$ is defined.
Since $\sigma \in \Phi_e(S; k)$, $S$ must have a prefix in $\Phi^{-1}_e(\sigma)$ so there exists
$x\in [S]\cap \dbra{H_{\sigma}}\subseteq [T']\cap \dbra{H_{\sigma}}$.

So $\exists \rho\in F,\ x\in M'_{\rho}(D_{\rho})$. But $x\in M'_{\rho}(D_{\rho})$ means
that $D_{\rho}=\sqbrad{\tau\in 2^k}{x\in\dbra{H_{\tau}}}$. Thus
$\sigma \in D_\rho$ because $x\in\dbra{H_{\sigma}}$, concluding the proof of \eqref{QnXMnFzaMQ}, which in turn gives
\begin{equation*}
\Phi_e(S;k)\subseteq \parb{2^k-C_k}\cup \parb{\cup_{\rho\in F} D_{\rho}}.
\end{equation*}
Since $\forall \rho\in F,\ |D_{\rho}|\leq 2^{k-n-e-1}$ we get 
\[
\abs{\cup_{\rho\in F} D_{\rho}}\leq \sum_{\rho\in F} \abs{D_{\rho}}
< 2^n\cdot 2^{k-e-n-1} = 2^{k-e-1}
\]
However by \eqref{eq:S3-div-k} we have $|C_k| \geq 2^k-2^{k-e-1}$ so 
\[
|\Phi_e(S;k)| \leq
\absb{\parb{2^k-C_k}\cup \parb{\cup_{\rho\in F} D_{\rho}}} <
|2^k-C_k| +  2^{k-e-1}\leq 2^{k-e} 
\]
which contradicts assumption about the size of the output of  $\Phi_e$.

Finally let $F' = F$.
Condition $(F', T')$ satisfies the requirements of the theorem.
\end{proof}
\begin{lem}\label{2MaJElyZz}
For each $e, (F,T)$ there exists $(F',T') < (F,T)$ such that:
\[
\forall S\in (F', T'):\  \parB{\Phi_e(S; k) \uparrow\ \vee \ \Phi_e(S; k)\de \not\subseteq P}.
\]
\end{lem}\begin{proof}
This is a consequence of Lemma \ref{BWZl9FShCb} and  Lemma \ref{5XvJQvkQxh}, 
since \eqref{5gEofWIwIJ}  allows the extension to a condition with robustness arbitrarily close to 1.
\end{proof}
We  conclude the proof of Theorem \ref{uZbmsMF7A6}. Given a positive tree $\Tb$, we may start with
condition $(F_0, T_0)=(\{\lambda\},\Tb)$ and consider a chain $(F_{e+1}, T_{e+1})<(F_{e}, T_{e})$ such that for each $e$
condition $(F', T'):=(F_{e}, T_{e})$ satisfies the property of Lemma \ref{2MaJElyZz}.
The chain $(F_{e}, T_{e})$ defines a perfect subtree $G$ of $\Tb$ such that no $(G\oplus z)$-\ce subset of $P$ is super-fat.

This  construction also proves Corollary \ref{OYzd34gN1g}.
Up to now,  $P$ was a fixed positive tree which has no $z$-\ce super-fat subset and the functionals $\Phi_e$ had an implicit oracle $z$.
Fixing $z=\emptyset$ and 
\[
P_e=\sqbrad{x}{\forall n\ K(x\restr_n)\geq n-e}
\]
Lemma \ref{2MaJElyZz} gives that
for each $e, (F,T)$ there exists $(F',T') < (F,T)$ with
\[
\forall S\in (F', T'):\  \parB{\Phi_e(S; k) \uparrow\ \vee \ \Phi_e(S; k)\de \not\subseteq P_e}.
\]
By the padding lemma (every Turing functional has infinitely many indices)  the chain $(F_e, T_e)$ with respect to this modification of 
Lemma \ref{2MaJElyZz} proves Corollary \ref{OYzd34gN1g}: there exists a perfect subtree of $\Tb$ which cannot computably enumerate
any super-fat subset of  $P_e$, for any $e$, hence any super-fat set of finite deficiency.

\subsection{The model for (b) of Theorem \ref{rms6rS57S4}}
We define a model for $\PPs+\neg \PPp$,
proving clause (b) of  Theorem \ref{rms6rS57S4}.
Let $P$ be the set of extendible nodes of a computable positive tree $P_0$ with a
uniform upper bound on the randomness deficiency of its infinite paths.
We a  sequence $z_n\leq_T z_{n+1}$ of reals  and  consider 
\begin{equation*}
\parbox{10cm}{the model $\II$ which is the  union of all $\II_n:=\sqbrad{z}{z\leq_T z_n}$.}
\end{equation*}
Starting with $z_0=\emptyset$, 
there exists no $z_0$-\ce super-fat subset of $P$, so there exists no $z_0$-computable positive perfect subtree of $P_0$.
Assume inductively that $z_i, i\leq n$ have been defined,
$z_i\leq_T z_{i+1}$ for $i<n$, and there exists no $z_n$-\ce super-fat subset of $P_0$.
\begin{itemize}
\item By Theorem \ref{uZbmsMF7A6} let $T_n$ be a perfect pathwise-$z_n$-random tree 
such that no  $(T_n\oplus z_n)$-\ce subset of $P$ is super-fat;
\item let $z_{n+1}:=T_n\oplus z_n$, so there exists no $z_{n+1}$-computable positive subtree of $P_0$.
\end{itemize}
If $x\in\II$, there exists no positive $x$-computable subtree of $P_0$,  so $\II$ does not contain any positive
subtree of $P_0$ and $P_0\in \II$. On the other hand, if $T\in \II$ is a  positive tree, there exists $n$ such that
$T\leq_T z_n$. Since $T_n$ is strongly pathwise $z_n$-random and perfect, by Lemma \ref{nqjFhRoRlv} there exists 
a perfect $T_n$-computable subtree of $T$, hence  a perfect subtree of $T$ in  $\II_{n+1}$.

\section{Positive trees not computing PA: Theorem \ref{xJFhzQCXsu}}\label{36Hjxwj7uW}
We start with a concise proof of a theorem of Patey, 
reported by \citet{Chong.Li.ea:2019}.
\begin{thm}[Patey]\label{thm:Patey_non-PA-perfect-subtree}
Every positive tree contains a perfect subtree which does not compute a
complete extension of PA. 
\end{thm}
This argument is given in \S\ref{GKlTOZxwFw} and then extended
in \S\ref{VhZ6RXMZ6v}, with aggressive use of Lebesgue density,
to a proof that there exists a positive perfect 
pathwise-random tree  which does not compute any complete extension of Peano Arithmetic.
Both arguments are inductive forcing constructions with conditions that are essentially closed
sets of reals. It is an an adaptation version of what is known as {\em Mathias forcing} to the space of trees. 

A total function $f$ is diagonally non-computable if 
 $\Phi_n(n) \downarrow\ \Rightarrow\ f(n) \neq \Phi_n(n)$, 
 where $(\Phi_n)$ is an effective enumeration of all Turing functionals, 
 here used without an oracle.
Recall that a real is $PA$ iff it computes a diagonally non-computable function which takes binary values
\citep[Theorem 6.9]{002945272010009}.
\subsection{Patey's theorem}\label{GKlTOZxwFw}
For the proof of Theorem \ref{thm:Patey_non-PA-perfect-subtree},
fix a positive tree $\hat{T}$.
We use the forcing notion in \S 3:
by Lebesgue density $(\{\lambda\}, \hat{T})$ is a condition.
We prove that if $(F_n,T_n)$ is a sufficiently generic sequence of conditions
and $T_0 = \hat{T}$ then 
$G = \cup_n F_n$ is a  perfect subtree of $\hat{T}$ of non-$PA$ degree.
\begin{lem}[Density of forcing conditions]\label{lem:Patey-non-dnc}
For each Turing functional $\Psi$, 
 there are densely many conditions forcing that $\Psi(G)$ is not a binary diagonally non-computable function.
\end{lem}

\begin{proof}
Assume that  $\Psi$ only takes binary values and $(F, T)$ is a condition.
By Lebesgue density, we may assume that $(F,T)$ is $3/4$-robust.
Define $\mathcal{U}$ to be the collection of trees $S$ such that
$\forall \sigma\in F,\ \mu_{\sigma}\big([S]\big) \geq 3/4$, and let $\tuple{}:\Nat\times\Nat\to\Nat$ be a computable bijection.
For each $n$ search for the least $\tuple{s,i}$ such that 
\begin{equation*}
\forall S \in \mathcal{U} \ \exists E \in \TTast \ 
\parB{F \preceq E \subset S \wedge \Psi_s(E; n) \downarrow = i}
\end{equation*}
and let $f(n) = i$; if the search does not halt, let $f(n)$ be undefined.  
Since $\mathcal{U}$ is effectively closed, $f$ is a partial $\Sigma^0_1$ function. So 
$f$ is not diagonally non-computable. 

If $f$ is total then $\exists n \ f(n)  = \Phi_n(n) \downarrow$. 
Let $S$ be the set of $\sigma$ such that $\mu_\sigma([T]) > 0$.
Then $S \in \mathcal{U}$ and $\mu_\sigma([S]) > 0$ for every $\sigma \in S$.
By the definition of $f$, 
 $S$ contains some $E \in \TTast$ 
 such that  $F \preceq E$ and $\Psi(E; n) \downarrow = f(n)$. 
So $(E,S)$ is an extension of $(F,T)$ forcing $\Psi(G; n) \downarrow = \Phi_n(n)$.

If $f$ is undefined on some $n$, fix such an $n$ 
and pick $S_0, S_1 \in \mathcal{U}$ such that
\begin{equation}\label{eq:forcing-non-PA-d}
\forall i < 2,\ \forall E \in \TTast:\  
\Big(F \preceq E \subset S_i \Rightarrow \Psi(E; n)\not\simeq i\Big).
\end{equation}
Let $S := T \cap S_0 \cap S_1$. 
By the definition of $\mathcal{U}$ and the fact that $S_0, S_1 \in \mathcal{U}$,
 we get  $\mu_{\sigma}([S]) > 2^{-2}$ for each leaf $\sigma$  of $F$. 
So $(F,S)$ is a condition extending $(F,T)$.
If $P$ is a perfect tree meeting $(F,S)$, by
\eqref{eq:forcing-non-PA-d} and the fact that $\Psi$ has binary values we get
$\Psi(P; n) \uparrow$. 
Hence $(F,S)$ forces that $\Psi(G; n) \uparrow$, which concludes the proof.
\end{proof}
Every condition $(F,T)$ can be extended to some $(E,T)$
 such that every element of $F$ has incomparable extensions in $E$.
This observation and Lemma \ref{lem:Patey-non-dnc} show that there exists a generic sequence of conditions 
which produces a tree satisfying the requirements of Theorem \ref{thm:Patey_non-PA-perfect-subtree}.
In relativized form, this argument shows that if $z$ is not $PA$ then 
every positive tree contains a perfect subtree $T$ such that $z\oplus T$ is not $PA$. 
The latter statement can be used inductively
for the construction of a model of $\RCA+\PPs+\neg\WKL$.

\subsection{Proof of Theorem \ref{xJFhzQCXsu}}\label{VhZ6RXMZ6v}
We extend the argument of \S\ref{GKlTOZxwFw} toward a proof of Theorem \ref{xJFhzQCXsu}:
given any  positive tree $T$ and $x\not\leq_T\emptyset$ 
there exists a  positive perfect subtree $P$ of $T$ which
does not have $PA$ degree and does not compute $x$.
A tree $T$ is  {\em strongly positive} if $\forall\sigma\in T,\ \mu_\sigma([T]) > 0$.

Without loss of generality we fix a strongly positive tree $T$
and construct the desired subtree $P$ of $T$ by forcing.
We control the density of $P$ above certain nodes
via positive rational functions whose domains are finite trees. 
We call such functions {\em lower density functions}, let
$\LDF$ denote their collection and
let $\tuple{d}$ denote the domain of $d \in \LDF$. 
If $d\in \LDF$ and $\epsilon>0$ is a rational,  
let $d+\epsilon$ denote the element of $\LDF$ with 
$\forall \sigma \in \tuple{d},\ (d+\epsilon)(\sigma) = d(\sigma) + \epsilon$.

\begin{defi}\label{def:dense-ext}
Given $d \in \LDF$ we say that
$E \in \TTast$ is \emph{weakly $d$-dense} if $\tuple{d} \preceq E$ and
\begin{equation*}%\label{eq:dense-ext}
d(\sigma)\cdot 2^{-|\sigma|} \leq \sum_{\sigma \preceq \tau \in E} 2^{-|\tau|} 
\hspace{0.5cm}\textrm{for all $\sigma \in \tuple{d}$.}
\end{equation*}
We say $E$ is \emph{$d$-dense} if it is weakly $(d+\epsilon)$-dense for a rational $\epsilon>0$.
An infinite tree $T$ is \emph{weakly $d$-dense}
 if $\tuple{d} \prec T$ and $T \cap 2^l$ is weakly $d$-dense for  sufficiently large $l$; equivalently:
\begin{equation*}%\label{eq:dense-ext-inf}
\tuple{d} \prec T
\hspace{0.5cm}\textrm{and}\hspace{0.5cm}
 \forall \sigma \in \tuple{d},  d(\sigma) \leq \mu_\sigma([T]). 
\end{equation*}
We say that $T$ is \emph{$d$-dense} if $T$ is weakly $(d+\epsilon)$-dense for some positive rational $\epsilon$.
\end{defi}

We define a relation on $\LDF$ which preserves the above density property.
Given $d, e\in \LDF$ we say that 
 {\em $e$ extends $d$} and denote it by $e \leq d$,
 if $\tuple{d} \preceq \tuple{e}$ and 
\[
\forall \sigma \in \tuple{d},\  d(\sigma) \cdot 2^{-|\sigma|} \leq \sum_{\sigma \preceq \tau \in \tuple{e}} e(\tau) \cdot 2^{-|\tau|}.
\]
Then if $e$ extends $d$,
every (weakly) $e$-dense $F \in \TTast$ or $T \in \TT$ is (weakly) $d$-dense.
Also $d+\epsilon$ extends $d$, where $\epsilon$ is any positive rational.

\begin{lem}\label{lem:dense-ext}
Let $d\in \LDF$.
For every $E \in \TTast$ and every strongly positive and $d$-dense tree $T$ 
 with $\tuple{d} \preceq E \prec T$  
 there exists  $e \in  \LDF$ such that $E = \tuple{e}$, $e \leq d$ and $T$ is $e$-dense.
\end{lem}
\begin{proof}
Let $\epsilon := \min\sqbrad{\mu_{\sigma}([T]) - d(\sigma)}{\sigma\in \tuple{d}} > 0$.

For each $\tau \in E$, let $e(\tau)$ be a rational such that
$e(\tau) < \mu_{\tau}([T]) < e(\tau) + \epsilon$.

It is clear that $e \in \LDF$ and $T$ is $e$-dense.
If $\sigma\in \tuple{d}$ then
\[
\mu_\sigma([T]) - \epsilon 
= \sum_{\sigma \preceq \tau \in E} \mu_\tau([T]) \cdot 2^{|\sigma| - |\tau|} - \epsilon
< \sum_{\sigma \preceq \tau \in E} (e(\tau) + \epsilon) \cdot 2^{|\sigma| - |\tau|} - \epsilon
\leq \sum_{\sigma \preceq \tau \in E} e(\tau) \cdot 2^{|\sigma| - |\tau|}
\]
and since $d(\sigma) \leq \mu_\sigma([T]) - \epsilon$ we get $e \leq d$.
\end{proof}

\begin{defi}\label{def:ppst-nPA_forcing}
\emph{A forcing condition} is a triple $p = (F_p, T_p, d_p)$ such that
\begin{itemize}
\item $d_p\in \LDF$ with domain $F_p \neq \emptyset$,
\item $T_p$ is an infinite $d_p$-dense and strongly positive tree.
\end{itemize}
Condition $p$ represents the set
$[p] := \{S \subseteq T_p: S \text{ is a strongly positive $d_p$-dense tree}\}$.

We say that $q$ {\em extends}  $p$ and denote it 
by $q \leq p$, if  $d_q \leq d_p$ and $T_q \subseteq T_p$. 
\end{defi}

Then $\leq$ is a partial ordering on conditions and $q \leq p\Rightarrow [q] \subseteq [p]$.
In order to see that the set of conditions is non-empty, 
 note that positive trees always have strongly positive subtrees.
So we may assume that the given tree $T$ is strongly positive.
If $F$ is the finite tree consisting of the empty string $\lambda$  and $d$ is a 
rational function with $F = \tuple{d}$, $d(\lambda)<\mu([T])$, then $(F,T,d)$ is a condition.
We show that  each condition $p$ has an extension $q$ such that the true 
density of $T_q$ is much closer to $1$ than the density specified by $d_q$.

\begin{lem}[Condensation]\label{lem:ppst-nPA_condensation}
For every $n$ and  condition $p$, there exists $q$ such that $q \leq p$ and
\[
\max\left\{ 1 - \mu_{\tau}([T_q]): \tau \in F_q \right\} < 
\frac{1}{n+1}\cdot \min\sqbrad{1 - d_q(\tau)}{\tau \in F_q}.
\]
\end{lem}

\begin{proof}
Let $l_p$ be the length of strings in $F_p$.
Since $T_p$ is $d_p$-dense, fix  rational $\delta>0$ with
\begin{equation}\label{hLW5EmDHPO}
\forall \sigma \in F_p,\  (1-\delta)\cdot \mu_\sigma([T_p])  > d_p(\sigma)
\end{equation}
and $\epsilon := \delta/(n+2)$.
By  Lebesgue density 
$\lim_{\sigma\prec x} \mu_\sigma([T_p])=1$ for almost all $x\in [T_p]$.

Combining this with \eqref{hLW5EmDHPO}  and the fact that $F_p$ is finite, 
 there exist $l_q \geq l_p$, $F_q \in \TTast$ with
\begin{enumerate}[(i)]
\item $F_p \preceq F_q \subset T_p \cap 2^{l_q}$\ \  \ and\ \ \  $\forall \tau \in F_q,\ \mu_{\tau}([T_p]) > 1-\epsilon$
\item   $\forall \sigma \in F_p,\ (1-\delta) \cdot |\{\tau \in F_q: \sigma \preceq \tau\}| \cdot 2^{-l_q} > d_p(\sigma) \cdot 2^{-l_p}$
\end{enumerate}
Let $d_q \in \LDF$ with  $\forall \tau \in F_q,\ d_q(\tau) = 1-\delta$.

By (ii), $d_q$ extends $d_p$. 
Let $T_q := \{\rho \in T_p: \rho \text{ is comparable with some } \sigma\in F_q\}$.
If $\tau \in F_q$, 
\[
  \mu_\tau([T_q]) = \mu_\tau([T_p]) > 1 - \epsilon > 1 - \frac{1}{n+1} \cdot \delta > 1 - \delta = d_q(\tau).  
\]
So $T_q$ is $d_q$-dense, and $q = (F_q, T_q, d_q)$ is a condition as desired.
\end{proof}

\begin{lem}[Branching]\label{lem:ppst-nPA_branching}
Every condition $p$ can be extended to some $q$ such that 
every node in $F_p$ has at least two incomparable extensions in $F_q$.
\end{lem}

\begin{proof}
Since $\forall \sigma \in F_p,\ \mu_{\sigma}([T_p]) > 0$ there exists $l$ so that every 
$\sigma \in F_p$ has at least two incomparable extensions in $T_p \cap 2^{l}$. 
By Lemma \ref{lem:dense-ext} there exists $e \in \LDF$ 
with $T_p \cap 2^{l} = \tuple{e}$, $e \leq d$, so
$q = (T_p \cap 2^{l}, T_p, e)$ is the required condition.
\end{proof}

Partial functions $f,g$ are {\em compatible}, denoted by $f\simeq g$,   if $f(n)=g(n)$ for each $n$ with $f(n)\de, g(n)\de$.

\begin{lem}[Cone Avoiding]\label{lem:ppst-nPA_ca}
For each Turing functional $\Psi$, each condition $p$ and each 
non-computable $x$, there exists $q \leq p$ such that $x \neq \Psi(S)$ for any $S \in [q]$.
\end{lem}\begin{proof}
By Lemma \ref{lem:ppst-nPA_condensation} we may assume that there exists a rational $\epsilon$ such that
\begin{equation}\label{VwbxUibqEy}
\sigma \in F_p
\hspace{0.3cm}\Rightarrow\hspace{0.3cm}
3\cdot \left( 1 -  \mu_{\sigma}([T_p]) \right) < 3\cdot \epsilon < 1 - d_p(\sigma).
\end{equation}
Then $T_p$ is $(d_p + 2 \cdot \epsilon)$-dense. 
Let $\mathcal{U}$ be the set of weakly $(d_p + 2 \cdot \epsilon)$-dense trees $S \in \TT$ such that
for any $S_0, S_1 \in \TT$: 
\[
S_i\ \text{are weakly $(d_p + \epsilon)$-dense subtrees of } S\  \Rightarrow\ \ \Psi(S_0)\simeq \Psi(S_1).
\]
We may identify $\mathcal{U}$  with a $\Pi^0_1$ class of Cantor space: the set of weakly $(d_p + \epsilon)$-dense finite trees is computable, thus the space of weakly $(d_p + \epsilon)$-dense infinite trees can be identified with a $\Pi^0_1$ class in Cantor space. The pairs $(S_0, S_1)$, such that both are $(d_p + \epsilon)$-dense and disagree via $\Psi$, form another computable set. So the weakly $(d_p + \epsilon)$-dense infinite trees, which do not contain such disagreeing pairs, form another $\Pi^0_1$ class,  denoted by $\mathcal{U}$.

The idea for the remaining argument is to look for a pair weakly dense finite subtrees of $T_p$ which computes a disagreement via $\Psi$. 
The existence of such pair is $\Sigma^0_1$ in $T_p$, thus too complicated for our purpose. 
Instead, we look through all infinite trees similar to $T_p$, in the space of all weakly $(d_p + \epsilon)$-dense infinite trees.
This trick is common in reverse mathematics of Ramsey theory  \cite{cholak2001strength, dzhafarov2009ramseys, wang2013rainbow}.
The trees in this space, which give a positive answer to the above existence question, form a $\Sigma^0_1$ class $V$, which is 
complement of $\mathcal{U}$. 
Then $V$ is the union of basic open sets determined by weakly dense finite trees containing a pair producing incompatible computations.
If all of the finite subtrees of $T_p$ contain a non-extendible node, the infinite subtree $S$ of $T_p$ obtained by removing all non-extendible nodes satisfies $\mu([S]) = \mu([T_p])$. 
Thus $S$ is in the space of weakly dense trees but not in $V$ and could serve as $S_0$ in Case 1 below.

{\em Case 1: $\mathcal{U} = \emptyset$.} 
In particular, $T_p \not\in \mathcal{U}$. 
Pick a weakly $(d_p+\epsilon)$-dense $S_0 \subseteq T_p$ in $\TT$, and $m$ 
with $\exists n,\ \Psi(S_0 \cap 2^{< m}; n) \downarrow \neq x(n)$. Let
\[
  T_q := (S_0 \cap 2^{< m}) \cup \sqbrad{\tau \in T_p}{\tau \text{ extends some } \sigma \in S_0 \cap 2^m}.
\]
Then $\forall \sigma\in F_p,\  \mu_\sigma([T_q]) \geq \mu_\sigma([S_0]) \geq  d_p(\sigma)+\epsilon$, so
$T_q$ is $d_p$-dense. 

By Lemma \ref{lem:dense-ext}, $\exists d_q\in \LDF$ with domain $F_q = S_0 \cap 2^m$ 
 such that  $d_q \leq d_p$ and $T_q$ is $d_q$-dense. 
Hence $q = (F_q,T_q,d_q)$ is a condition extending $p$ and 
 $\forall S \in [q], \Psi(S; n) \downarrow \neq x(n)$.

{\em Case 2: $\mathcal{U} \neq \emptyset$.} By the cone-avoidance property of $\Pi^0_1$ classes
\citep[Theorem 4.5]{002945272010009} we may  
pick $S \in \mathcal{U}$ such that $x\not\leq_T S$. 
Then
\[
  T_q := \sqbrad{\sigma}{\mu_\sigma([S \cap T_p]) > 0}
\]
is strongly positive. By $S\in \mathcal{U}$, the definition of $\mathcal{U}$ and
\eqref{VwbxUibqEy}, $T_q$ is also
$(d_p + \epsilon)$-dense.
 
So $q = (F_p, T_q, d_p+\epsilon) \leq p$ and 
every $R \in [q]$ is a $(d_p+\epsilon)$-dense subtree of $S$. 

Then if $\Psi(R)$ is total, $\Psi(R)\leq_T S$ and thus $\Psi(R)\not\simeq x$.
\end{proof}

\begin{lem}[Non-PA]\label{lem:ppst-nPA_ndnc}
For each Turing functional $\Psi$ and each condition $p$, 
 there exists $q \leq p$ 
 such that if $S \in [q]$ and $\Psi(S)$ is total then $\Psi(S; n) = \Phi_n(n) \downarrow$ for some $n$.
\end{lem}

\begin{proof}
By Lemma \ref{lem:ppst-nPA_condensation},
 we may assume that there exists a rational $\epsilon$ such that
\begin{equation}\label{kknui7BK8J}
\sigma\in F_p
\hspace{0.3cm}\Rightarrow\hspace{0.3cm}
4\cdot \left( 1 -  \mu_\sigma([T_p]) \right) < 4\cdot \epsilon < 1 - d_p(\sigma).
\end{equation}
Define $d'\in \LDF$ with domain  $F_p$ and $d'(\sigma) := 1-\epsilon$.
Then $T_p$ is $d'$-dense. 

Define a partial function $f$ as follows. 
For each $n$, search for the least pair $(l,i)$ for which we notice that 
each weakly $d'$-dense $S \in \TT$ contains a $(d_p+\epsilon)$-dense $E \in \TTast$ satisfying
\begin{equation}\label{eq:ppst-nPA_E}
 E \subseteq 2^{l}\ \  \wedge\ \  \Psi(\{\sigma: \sigma \text{ has an extension in } E\}; n) \downarrow = i.
\end{equation}
If such a pair $(l,i)$ is found then let $f(n) = i$, otherwise $f(n)\un$. 
Since the class of weakly $d'$-dense elements of $\TT$ can be viewed as a $\Pi^0_1$ class, 
 it follows that $f$ is $\Sigma^0_1$. 
Hence $f$ is not diagonally non-computable.

{\em Case 1: $f$ is total.}  
Fix $n$ such that $f(n) = \Phi_n(n) \downarrow$. 
As $T_p$ is $d'$-dense, $T_p$ contains a $(d_p+\epsilon)$-dense $E$ satisfying \eqref{eq:ppst-nPA_E}. Let
\[
T_q = E \cup \sqbrad{\sigma \in T_p}{\sigma \text{ is comparable with some } \tau \in E}.
\]
Then $E \prec T_q$ and $T_q$ is $d_p$-dense. 
By Lemma \ref{lem:dense-ext}, $d_p$ admits an extension $d_q$ 
 such that $\tuple{d_q} = E$ and $T_q$ is $d_q$-dense. 
Let $q = (E,T_q,d_q)$, which is clearly an extension of $p$. 
If $S \in [q]$ then $\Psi(S; n) \downarrow = \Phi_n(n)$.

{\em Case 2: $\exists n\ f(n)\un$.} By compactness, 
for $i < 2$ there exists  a weakly $d'$-dense $S_i \in \TT$ with 
\[
  \Psi(\{\sigma: \sigma \text{ has an extension in } E\}; n)  \not\simeq i 
\]
for each $(d_p+\epsilon)$-dense subtree $E \in \TTast$  of $S_i$.
Then
\[
T_q := \{\sigma: \mu_\sigma([S_0 \cap S_1 \cap T_p]) > 0\}
\]
is strongly positive.
Since $d'(\sigma) := 1-\epsilon$ and each $S_i$ is weakly $d'$-dense, $\mu_{\sigma}(S_0\cap S_1)\geq 1-2\epsilon$.

By \eqref{kknui7BK8J}, for each $\sigma\in \tuple{p}$
\[
\mu_{\sigma}([T_q])=\mu_{\sigma}([S_0 \cap S_1 \cap T_p])\geq 1-3\epsilon> d_p+\epsilon
\]
so $T_q$ is $(d_p+\epsilon)$-dense and
$q = (F_p,T_q,d_p+\epsilon)$ is an extension of $p$. 

Every $S \in [q]$ is a $(d_p+\epsilon)$-dense subtree of $S_0 \cap S_1$. 
By the choice of the $S_i$, $\Psi(S; n) \uparrow$.
\end{proof}

We conclude the proof of Theorem \ref{xJFhzQCXsu}. Let $(\Psi_i)$
be an enumeration of all Turing functionals and fix a positive tree $T$ and a non-computable real $x$. 
By Lemmata \ref{lem:ppst-nPA_branching}, \ref{lem:ppst-nPA_ca} and \ref{lem:ppst-nPA_ndnc} 
there exists a descending sequence $(p_i)$ of conditions such that
$T_{p_0} = T\ \wedge\ d_{p_0}(\emptyset) > 0$ and:
\begin{enumerate}[\hspace{0.3cm}(a)]
% \item $T_{p_0} = T\ \wedge\ d_{p_0}(\emptyset) > 0$
\item $\forall i\ \exists j>i: $  every $\sigma \in F_{p_i}$ has two incomparable extensions in $F_{p_j}$,
\item $\forall i\ \exists j>i\ \forall S \in [p_j]:  \parB{\Psi_n(S) \neq x\ \wedge\ \textrm{$\Psi_n(S)$ is not diagonally non-computable}}$.
\end{enumerate}
Let $P := \bigcup_i F_{p_i} \in \bigcap_i [p_i]$.
By Lemma \ref{lem:ppst-nPA_branching}, $P$ is perfect.
Since every $F_{p_i}$ is weakly $d_{p_0}$-dense, $P$ is weakly $d_{p_0}$-dense as well.
Thus $P$ is a positive perfect subtree of $T$. Hence $x \not\leq_T P$ and  by (b), $P$ is not PA.

{\bf Model: positive avoiding PA.}
We may now use Theorem \ref{xJFhzQCXsu} in relativized form in order to obtain an $\omega$-model of
$\RCA+\PPp+\neg\WKL$, proving
clause (c) of  Theorem \ref{rms6rS57S4}.
Let $P$ be a non-empty \pz class of reals
with a uniform bound on their randomness deficiency.
A direct relativization of the above argument to a non-$PA$ real $z$ gives the following:
\begin{equation}\label{mK2oHOYzd}
\parbox{13cm}{if $z$ is not $PA$:
$\exists$  perfect positive pathwise-$z$-random tree $T$ such that $z\oplus T$ is not $PA$.}
\end{equation}
We  define  $(z_n)$  where $z_0$ is computable, $\forall n, z_n\leq_T z_{n+1}$ and  consider 
the model $\II$ which is the  the union of all $\II_n:=\sqbrad{z}{z\leq_T z_n}$.
Inductively assume that $z_i, i\leq n$
are defined such that  $z_{i}\leq_T z_{i+1}$ for $i<n$  and
$z_n$ is not $PA$.  We use \eqref{mK2oHOYzd} in order to obtain a positive pathwise-$z_n$-random perfect tree $T_n$ 
such that $z_n\oplus T_n$ is not $PA$, and let $z_{n+1}:=z_n\oplus T_n$.
Clearly $\II$ does not contain any $PA$ reals. Given any positive tree $T$ in $\II$, there exists $n$ such that $T\in \II_n$
and by Lemma \ref{nqjFhRoRlv}, there exists a positive subtree $T'$ of $T$ in $\II_{n+1}$, hence in $\II$.

\section{Conclusion}\label{FWznALqzcD}
We have studied the pathwise-random trees and demonstrated their utility in
calibrating compactness principles in second-order arithmetic.
We pointed out that this type of randomness on trees is quite different from the notions of algorithmically random 
branching processes studied in \citep{algClosedBarmp, DKH101007,DIAMOND2012519}, and closer to the effective
Poisson point processes studied in \citep{Axonphd,axon2015}. Intuition about hitting probabilities in such processes
informed the constructions in \S 2, \S 3. 

We studied three principles between $\WKL$ and $\WWKL$:
every positive  tree has a positive perfect subtree, a 
perfect subtree, and an infinite countable family of paths, denoted by $\PPp$, $\PPs$, $\PPm$ respectively.
We separated $\WWKL$ from $\PPm$ (over $\RCA$) via a model generated by random reals. 
Separations of $\PPs$ from  $\PPp$ and $\PPp$ from $\WKL$ were achieved  via forcing arguments. 
It would be interesting to know if some type of algorithmic stochastic process produces
perfect pathwise-random trees that do not compute positive trees of finite deficiency, which are the
building blocks for separating $\PPs$ from $\PPp$. The separation of  $\PPm$ from $\PPs$ is left open.

%\bibliographystyle{abbrvnat}
%\bibliography{treerapa}
\end{document}